\documentclass[11pt,reqno]{amsart}
\usepackage[utf8]{inputenc}
\usepackage[a4paper, hmargin={2.7cm,2.7cm},vmargin={3.3cm,3.3cm}]{geometry}
\usepackage[english]{babel}
\usepackage{color}
\usepackage{amsmath}
\usepackage{amssymb}
\usepackage{amsfonts}
\usepackage{amsthm}
\usepackage{enumerate}
\usepackage{mathrsfs}
\usepackage{amsthm}
\usepackage[noadjust]{cite}
\usepackage{dsfont}
\usepackage{etoolbox}
\usepackage[unicode=true]{hyperref}
\usepackage[noabbrev]{cleveref}
\usepackage{esint}
\usepackage{todonotes}
\usepackage{csquotes}
\MakeOuterQuote{"}
\usepackage{enumitem}
\usepackage{todonotes}
\usepackage{comment}

\binoppenalty=\maxdimen
\relpenalty=\maxdimen
\makeatletter
\patchcmd{\ttlh@hang}{\parindent\z@}{\parindent\z@\leavevmode}{}{}
\patchcmd{\ttlh@hang}{\noindent}{}{}{}
\makeatother

\makeatletter
\@namedef{subjclassname@2020}{\textup{2020} Mathematics Subject Classification}
\makeatother

\theoremstyle{plain}
\newtheorem{theorem}{Theorem}[section]
\newtheorem{lemma}[theorem]{Lemma}
\newtheorem{proposition}[theorem]{Proposition}
\newtheorem{corollary}[theorem]{Corollary}

\newtheorem{innercustomgeneric}{\customgenericname}
\providecommand{\customgenericname}{}
\newcommand{\newcustomtheorem}[2]{%
  \newenvironment{#1}[1]
  {%
   \renewcommand\customgenericname{#2}%
   \renewcommand\theinnercustomgeneric{##1}%
   \innercustomgeneric
  }
  {\endinnercustomgeneric}
}

\def\XXint#1#2#3{{\setbox0=\hbox{$#1{#2#3}{\int}$ }
\vcenter{\hbox{$#2#3$ }}\kern-.6\wd0}}

\newcustomtheorem{customthm}{Theorem}
\newcustomtheorem{customlemma}{Lemma}

\theoremstyle{definition}
\newtheorem{definition}[theorem]{Definition}
\newtheorem{example}[theorem]{Example}

\newtheorem{notation}[theorem]{Notation}

\theoremstyle{remark}
\newtheorem*{remark}{Remark}

\numberwithin{equation}{section}

\DeclareMathOperator*{\supp}{supp}

\DeclareMathOperator*{\esssup}{ess\,sup}

\DeclareMathOperator*{\Co}{Co}

\usepackage{xparse}

\newcommand{\Schwartz}{\mathcal{S}}

\newcommand{\vertiii}[1]{{\left\vert \kern-0.25ex
                            \left\vert \kern-0.25ex
                              \left\vert #1\right\vert\kern-0.25ex
                            \right\vert \kern-0.25ex
                          \right\vert}}

\newcommand{\eps}{\varepsilon}
\renewcommand{\emptyset}{\varnothing}

\newcommand{\CalQ}{\mathcal{Q}}

\newcommand{\CalP}{\mathcal{P}}

\newcommand{\CalO}{\mathcal{O}}

\newcommand{\GL}{\operatorname{GL}}
\newcommand{\Fourier}{\mathcal{F}}

\DeclareFontFamily{U}{mathx}{\hyphenchar\font45}
\DeclareFontShape{U}{mathx}{m}{n}{
      <5> <6> <7> <8> <9> <10>
      <10.95> <12> <14.4> <17.28> <20.74> <24.88>
      mathx10
      }{}
\DeclareSymbolFont{mathx}{U}{mathx}{m}{n}
\DeclareFontSubstitution{U}{mathx}{m}{n}
\DeclareMathAccent{\widecheck}{0}{mathx}{"71}
\DeclareMathAccent{\wideparen}{0}{mathx}{"75}

\newcommand{\R}{\mathbb{R}}

\newcommand{\CC}{\mathbb{C}}
\newcommand{\N}{\mathbb{N}}
\newcommand{\Z}{\mathbb{Z}}

\title[Wavelet coorbit spaces associated to different dilations]{On wavelet coorbit spaces  associated to \\ different dilation groups}

\author{Hartmut F\"uhr}

\address{Lehrstuhl f\"ur Geometrie und Analysis, RWTH Aachen University, D-52056 Aachen,
Germany}

\email{fuehr@mathga.rwth-aachen.de}

\author{Jordy Timo van Velthoven }

\address{Faculty of Mathematics,
University of Vienna,
Oskar-Morgenstern-Platz 1,
A-1090 Vienna, Austria}

\email{jordy.timo.van.velthoven@univie.ac.at}

\author{Felix Voigtlaender}

\address{Mathematical Institute for Machine Learning and Data Science (MIDS),
Catholic University of Eichstätt–Ingolstadt (KU),
Auf der Schanz 49, 85049 Ingolstadt, Germany}

\email{felix.voigtlaender@ku.de}



\begin{document}

\maketitle

\begin{center}
    \emph{Dedicated to Karlheinz Gröchenig on the occasion of his 65th birthday}
\end{center}

\begin{abstract}
This paper develops methods based on coarse geometry for the comparison of wavelet coorbit spaces
defined by different dilation groups, with emphasis on establishing a unified approach to both irreducible and reducible quasi-regular representations.
We show that the use of reducible representations is essential to include a variety of examples, 
such as anisotropic Besov spaces defined by general expansive matrices, in a common framework.

The obtained criteria yield, among others, a simple characterization of subgroups
of a dilation group yielding the same coorbit spaces. They also allow to clarify which anisotropic Besov spaces have an alternative description as coorbit spaces associated to irreducible quasi-regular representations. 
\end{abstract}

\section{Introduction}
\label{sec:Introduction}

The theory of coorbit spaces was initially developed by Feichtinger and Gröchenig
\cite{MR0942257, feichtinger1989banach2, feichtinger1989banach1, grochenig1991describing}
as a group-theoretic framework that allows to view various classical function spaces
in complex and harmonic analysis, such as Bergman, Bargmann-Fock and Besov-Triebel-Lizorkin spaces,
under a unified perspective. 
These initial papers have resulted in an extensive body of literature,
studying generalizations in various directions, and demonstrating the applicability
of the coorbit method in a large variety of settings, see, e.g., the papers
\cite{MR2776571,christensen2011coorbit,rauhut2007coorbit,christensen2019coorbits,dahlke2008generalized,%
velthoven2022coorbit,fornasier2005continuous,UllrichSchaeferGeneralizedCoorbitTheory}
and the references therein. 

A particularly rich source of examples within the original setting of coorbit spaces
that have been studied extensively in the past few years are function spaces on Euclidean space
that are invariant under translations and certain matrix dilations.
To be explicit, we let $H \leq \mathrm{GL}(d, \R)$ be a closed subgroup
and consider the semidirect product group $G = \R^d \rtimes H$.
The \emph{quasi-regular representation} $\pi$ of $G$ is the unitary representation on $L^2 (\R^d)$
defined by
\begin{align} \label{eq:quasiregular_intro}
  \pi(x, h) f(t)
  = |\det h|^{-\frac{1}{2}} f(h^{-1}(t - x)),
  \quad t \in \R^d. 
\end{align}
In case $\pi$ is irreducible and square-integrable, the associated coorbit spaces
are defined by imposing norm conditions on the wavelet transform
$W_{\psi} f = \langle f, \pi (\cdot) \psi \rangle$ of a function/distribution $f$
and an adequate fixed function $\psi$.
For example, given  $1 \leq p \leq 2$, the coorbit space $\Co(L^p(G))$ is the Banach space
of functions $f \in L^2 (\R^d)$ satisfying 
\begin{align} \label{eq:coorbit_intro}
  \| f \|_{\Co(L^p)}^p
  := \int_{G} | \langle f, \pi(x,h) \psi \rangle |^p \; d\mu_G (x,h)
  < \infty,
\end{align}
where $\mu_G$ denotes the left Haar measure on $G$.
For the similitude group $H = \R^+ \cdot \mathrm{SO}(d)$, the coorbit spaces $\Co(L^p(\R^d \rtimes H))$
correspond to classical (homogeneous) Besov spaces \cite{peetre1976new, frazier1985decomposition}
on $\R^d$ (see, e.g., \cite{groechenig1988unconditional, grochenig1991describing, MR0942257}), 
and the general coorbit theory \cite{MR0942257,feichtinger1989banach2,feichtinger1989banach1,grochenig1991describing}
revealed that the classical atomic decompositions of such spaces can be obtained
as consequences of the action of the group on these spaces.
Just like the classical Besov space $\dot{\mathrm{B}}^{0}_{1,1} (\R^d)$ is (in a certain sense)
the minimal Banach space being invariant under translations and dilations (cf.\ \cite{frazier1991littlewood}),
the coorbit space $\Co(L^1(G))$ can be shown to be the minimal Banach space
that is invariant under the action of the quasi-regular representation \eqref{eq:quasiregular_intro}.
As such, for other adequate dilation groups $H$, the spaces $\Co(L^p(\R^d \rtimes H))$
can be shown to yield new classes of function spaces, which have been studied intensively in,
e.g.,  \cite{MR3345605,MR3378833,CoarseGeometryPaper, AchimPhD, fuhr2017simplified, fuhr2023dilational}.

The present paper is concerned with two questions related to coorbit spaces defined
by the quasi-regular representations \eqref{eq:quasiregular_intro}
that can be traced back to the very beginnings of coorbit space theory,
namely the possibility of using \emph{reducible} quasi-regular representations
to define coorbit spaces and the question of when two different dilation groups
yield the same coorbit space.

\subsection{Previous work}

Before describing the content of the present paper in more detail, we describe the relevant
context and some recent developments.

First, although the original coorbit space theory assumed the group representation
to be \textit{irreducible}, that is, $\{0\}$ and $L^2 (\R^d)$ are the only
closed $\pi$-invariant subspaces of $L^2 (\R^d)$,
it was already realized from early on that an integrable reproducing kernel
(rather than irreducibility) is the essential assumption
guaranteeing most properties of coorbit spaces,
see, e.g., \cite[Remark 6.6]{grochenig1991describing}.
This was one of the key motivations for the study of a large class of dilation groups,
the so-called \emph{integrably admissible dilation groups} (cf.\ \Cref{sec:integrably}),
that guarantee an integrable quasi-regular representation; see \cite{currey2016integrable}.
Various concrete aspects of the coorbit spaces associated with such dilation groups
were studied in \cite{FuehrVelthoven}, and allowed to incorporate further scales of function spaces,
such as anisotropic Besov spaces \cite{bownik2005atomic}, into a common framework
with the already established coorbit spaces defined by irreducible,
square-integrable quasi-regular representations studied in,
e.g., \cite{MR3345605,MR3378833,CoarseGeometryPaper}.

Second, as already outlined above, it was realized early in the development of coorbit space theory
that the scale of (homogeneous, isotropic) Besov spaces in arbitrary dimensions could be understood as
coorbit spaces associated to the irreducible quasi-regular representation of the similitude group
$H_1 = \R^+ \cdot \mathrm{SO}(d)$.
However, it was also understood early on that with suitable choices of analyzing wavelets,
the full similitude group $H_1$ could be replaced by the one-parameter group
$H_2 = \mathbb{R}^+ \cdot I$ and yield the same coorbit space,
see, e.g., \cite[Remark (ii)]{groechenig1988unconditional}. 
This led naturally to the question of which dilation groups
yield different scales of coorbit spaces.
For dilation groups that act irreducibly, the paper \cite{CoarseGeometryPaper}
provided a coarse geometric approach to this question, and showed that whether two dilation groups
$H_1, H_2$ yield the same coorbit space can be decided by investigating
whether a suitable map between $H_1$ and $H_2$ is a quasi-isometry.

\subsection{Aims and contributions}

The main purpose of the present paper is to
extend the methods based on coarse geometry for the comparison of coorbit spaces developed in \cite{CoarseGeometryPaper} 
to the general setting of integrably admissible dilation groups \cite{FuehrVelthoven}, and thus obtain far-reaching criteria for the characterization of coorbit spaces and their symmetries.
In doing so, we resolve an open question on the uniqueness of a dual orbit
or essential frequency support of an integrably admissible dilation group
(cf.\ \cite[Section 2.1]{FuehrVelthoven}).
As an application and illustration of our results, we show that the inclusion of \textit{reducible}
integrably admissible dilation groups is necessary for coorbit space theory to cover anisotropic Besov spaces \cite{bownik2005atomic}: It was already known that these spaces are coorbit spaces with respect to one-parameter dilation groups \cite{FuehrVelthoven}, which necessarily act reducibly. Our Theorem \ref{thm:oneparameter} provides a complement to that observation by stating that these spaces do not have a characterization using an irreducibly acting dilation group. 

\subsection{Technical overview}

As in \cite{CoarseGeometryPaper}, our approach towards the comparison of coorbit spaces
hinges on the description of such a space as a \emph{Besov-type space}
defined by decomposition methods
\cite{triebel1977general,strockert1979decomposition,triebel1978spaces},
also called a \emph{Besov-type decomposition space}
\cite{feichtinger1985banach, VoigtlaenderEmbeddingsOfDecompositionSpaces}.
In contrast to the usual coorbit space norm (see, e.g., \Cref{eq:coorbit_intro}),
the Besov-type spaces are defined by means of a \emph{discrete} Littlewood-Paley-type norm.
For example, for $1 \leq p \leq 2$, the coorbit space $\Co(L^p(G))$ can alternatively be described
as those $f \in L^2 (\R^d)$ satisfying a norm condition of the form
\[
  \sum_{i \in I}
    \Bigl( |\det (h_i)|^{\frac{1}{2} - \frac{1}{p}} \cdot \big\| f \ast \varphi_i \big\|_{L^p} \Bigr)^p
  < \infty,
\]
where $(\varphi_i)_{i \in I}$ is such that the family of Fourier transforms 
$(\widehat{\varphi_i})_{i \in I}$ forms a suitable partition of unity
adapted to a frequency cover $\CalQ = (Q_i)_{i \in I}$ of the form $Q_i = h_i^{-T} Q$
for some discrete family $(h_i)_{i \in I}$ of points $h_i \in H$.
The realization of a coorbit space as a Besov-type space has been shown for irreducibly acting
dilation groups in \cite{MR3345605}, and for general integrably admissible dilation groups
in \cite{FuehrVelthoven}. This identification is an essential ingredient in our approach.

The significance of Besov-type spaces for our purposes is that a comparison of these spaces
in terms of geometric properties of associated frequency covers has been obtained in
\cite{VoigtlaenderEmbeddingsOfDecompositionSpaces}.
More precisely, the classification results in \cite{VoigtlaenderEmbeddingsOfDecompositionSpaces}
show that the coincidence of a scale of Besov-type spaces (hence, wavelet-type coorbit spaces)
is equivalent to the \emph{weak equivalence} (cf.\ \Cref{def:weakequivalence})
of the frequency covers defining the spaces.
The paper \cite{CoarseGeometryPaper} showed in turn that the weak equivalence of two covers
is equivalent to the ambient spaces being quasi-isometric relative to two so-called
\emph{cover-induced metrics} (see \Cref{thm:WeakEquivQuasiIso}).
This latter condition allows, in the case of covers induced by dilation groups,
to characterize when two dilation groups $H_1$ and $H_2$ induce the same coorbit spaces
in terms of a quasi-isometry between $H_1 \times C_1$ and $H_2 \times C_2$
for certain compact sets $C_1 \subseteq \R^d \setminus \{ 0 \}$ and $C_2 \subseteq \R^d \setminus \{ 0 \}$
(see \Cref{thm:char_coorbit_equiv_new}).

In comparison to coorbit spaces associated to irreducibly admissible dilation groups
in \cite{CoarseGeometryPaper}, there are several additional difficulties that arise
in the general setting of integrably admissible dilation groups treated in the present paper.
First of all, in contrast to an irreducibly acting dilation group, it is \emph{a priori}
not clear that an integrably admissible dilation group admits
a \emph{unique} associated essential frequency support (see \Cref{sec:integrably}).
This question is, however, of fundamental importance for our approach
towards comparing coorbit spaces, because the essential frequency support
enters crucially in the description of a coorbit space as a Besov-type space.
We show the uniqueness of the essential frequency support in \Cref{sec:unique_frequencysupport}.
Second, the essential frequency support is in general not singly generated
(i.e., the action $(h,\xi) \mapsto h^T \xi$ on the essential frequency support
is in general not transitive),
in contrast to the special case of an irreducible admissible dilation group.
As in \cite{FuehrVelthoven}, the fact that the frequency support is possibly generated
by an arbitrary compact set leads to various technicalities
(e.g., uniformity of certain estimates), as we point out throughout the text.

\subsection{Organization}

The paper is organized as follows: \Cref{sec:integrablyadmissible} recalls various properties
of integrably admissible dilation groups that will be used throughout the text.
In addition, the uniqueness of a frequency support of such a group
is shown in \Cref{sec:unique_frequencysupport}.
\Cref{sec:ConnectivityRespectingGroups} is devoted to showing that the \emph{full orbit map}
\[
  p :  H \times C \to \CalO, \quad (h,\xi) \mapsto h^{-T} \xi
  ,
\]
with $C \subseteq \CalO$ compact and $\CalO = H^T C$,
forms a quasi-isometry, under suitable (mild) technical assumptions,
and using suitable natural metrics on $H$ and $\CalO$.
Before proving this in \Cref{sec:coarseequivalence_proof}, all requisite background and notions
are recalled in the prior subsections.
The results from \Cref{sec:integrablyadmissible} and \Cref{sec:ConnectivityRespectingGroups}
are combined in \Cref{sec:coorbit_equivalence} to characterize when two dilation groups
yield the same coorbit space.
Lastly, the results of \Cref{sec:coorbit_equivalence} are applied in \Cref{sec:application}
to study anisotropic Besov spaces \cite{bownik2005atomic}
inside the framework of coorbit space theory.
The appendices contain some postponed proofs and auxiliary  results used in the main text.

\subsection*{Notation} The set of natural numbers (excluding $0$) is denoted by $\mathbb{N}$.
For two functions $f_1, f_2 : X \to [0, \infty)$ on a set $X$,
we write $f_1 \lesssim f_2$ if there exists $C > 0$ such that $f_1 (x) \leq C f_2 (x)$ for all $x \in X$.
We use the notation $f_1 \asymp f_2$ whenever $f_1 \lesssim f_2$ and $f_2 \lesssim f_1$.

\section{The essential frequency support of integrably admissible dilation groups}
\label{sec:integrablyadmissible}

This section recalls the class of integrably admissible dilation groups
studied in \cite{currey2016integrable, FuehrVelthoven}.
For such groups, we show that there exists at most one associated essential frequency support,
which plays a key role in establishing our main results on the classification of coorbit spaces.

\subsection{Integrably admissible dilation groups}
\label{sec:integrably}

Following \cite{currey2016integrable, FuehrVelthoven}, we say that a closed subgroup
$H \leq \mathrm{GL}(d,\R)$ is \emph{integrably admissible},
if there exists a set $\CalO \subseteq \R^d$ satisfying the following conditions:
\begin{enumerate}[label=(a\arabic*)]
  \item \label{enu:FrequencySupportFullMeasure}
        The set $\CalO \subseteq \R^d$ is open and of full measure (i.e., $\CalO^c$ is a null-set);

  \item \label{enu:FrequencySupportCompactlyGenerated}
        there exists a compact set $C \subseteq \CalO$ such that $\CalO = H^T C$;

  \item \label{enu:FrequencySupportProperness}
        for each compact set $K \subseteq \CalO$, the closed set
        \begin{align*}
          [K]
          & := \big\{
                 (h,\xi) \in H \times \CalO
                 \colon
                 (h^T \xi, \xi) \in K \times K
               \big\} 
          \subseteq H \times \CalO
        \end{align*}
        is compact.
 \end{enumerate}
Every set $\CalO \subseteq \R^d$ satisfying the conditions \ref{enu:FrequencySupportFullMeasure},
\ref{enu:FrequencySupportCompactlyGenerated} and \ref{enu:FrequencySupportProperness}
is referred to as an \emph{(essential) frequency support} associated to the dilation group $H$. 

The significance of the class of integrably admissible dilation groups
for our purposes is that it guarantees that the quasi-regular representation \eqref{eq:quasiregular_intro}
admits an admissible vector with an integrable matrix coefficient;
see \Cref{prop:CcInftyAdmissibleVectorExists}.
Both properties are used in the coorbit theory developed in \Cref{sec:coorbit_equivalence}. 

For two sets $U, V \subseteq \R^d$, we define the set
\begin{equation}
  ((U, V)) := \{ h \in H : h^T U \cap V \neq \emptyset \}.
  \label{eq:SpecialSetDefinition}
\end{equation}
The following characterization of integrably admissible dilation groups
is often used in the remainder.
See \cite[Proposition 2.9]{currey2016integrable} for its proof
(and also see \cite[Definition 2.2]{currey2016integrable}).

\begin{lemma}[\cite{currey2016integrable}] \label{lem:integrably_characterization}
    Let $H \leq \mathrm{GL}(d, \R)$ be a closed subgroup and let $\CalO \subseteq \mathbb{R}^d$
    be an open set that is $H^T$-invariant and of full measure.
    Then the following assertions are equivalent:
    \begin{enumerate}[label=(\roman*)]
        \item $H$ is integrably admissible with essential frequency support $\CalO$;

        \item There exists some open, relatively compact set $C \subseteq \CalO$ with $\overline{C} \subseteq \CalO$ and $H^T C = \CalO$,
              and for any such set $C$, the set $((C, C))$ is relatively compact in $H$.
    \end{enumerate}
\end{lemma}

We collect the following consequences of \Cref{lem:integrably_characterization};
see \cite[Lemma 2.4]{currey2016integrable} and \cite[Lemma 2.5]{currey2016integrable}.

\begin{corollary}[\cite{currey2016integrable}] \label{cor:basic_integrably}
Let $H \leq \mathrm{GL}(d, \R)$ be integrably admissible with essential frequency support $\CalO$.
Then the following assertions hold:
\begin{enumerate}[label=(\roman*)]
  \item For each $\xi \in \CalO$, the stabilizer subgroup
        $H_{\xi} := \{ h \in H : h^T \xi = \xi \}$ is compact.

  \item For arbitrary compact sets $C_1, C_2 \subseteq \CalO$, the set $((C_1, C_2))$ is compact.
\end{enumerate}
\end{corollary}

In the following example, two classes of integrably admissible dilation groups are listed
that have been studied in the literature before, see, e.g.,
\cite{bruna2015characterizing,laugesen2002characterization,currey2016integrable,fuhr2010generalized,MR1151331}.

\begin{example}[Integrably admissible dilation groups]\label{ex:dilations}
  Let $H \leq \mathrm{GL}(d, \R)$ be closed.
  \begin{enumerate}[label=(\arabic*)]
    \item $H$ is called an \emph{irreducible admissible dilation group} if there exists
          an open singly generated orbit $\mathcal{O} := H^{T} \xi_0 \subseteq \R^d$
          of full measure for some $\xi_0 \in \R^d$
          for which the isotropy group $H_{\xi_0}$ is compact.
          The irreducibly admissible dilation groups are precisely those for which the quasi-regular 
          representation on is irreducible, cf. \cite[Corollary 21]{fuhr2010generalized}. 
          Any irreducibly admissible dilation group is integrably admissible.

    \item A one-parameter subgroup $H = \exp(\R A)$ is integrably admissible if and only if
          the real parts of all eigenvalues of $A$ are either
          strictly negative or strictly positive; see \cite[Theorem 1.1]{MR1151331}
          and \cite[Proposition 6.3]{FuehrVelthoven}.
  \end{enumerate}
\end{example}

\subsection{Essential frequency support}\label{sec:unique_frequencysupport}

The aim of this section is to prove that there exists at most one essential frequency support
for an integrably admissible dilation group.
This question was discussed, but left open, in \cite[Section 2]{FuehrVelthoven}. 

We first prove a lemma that allows us to reduce the general question
to the case where different frequency supports fulfill a containment relation. 

\begin{lemma}\label{lem:DualOrbitUnion}
  Let $H \leq \GL(d, \R)$ be integrably admissible
  with associated essential frequency supports $\CalO_1, \CalO_2$. 
  Then $\CalO := \CalO_1 \cup \CalO_2$ is an essential frequency support associated to $H$.
\end{lemma}

\begin{proof}
  Clearly, $\CalO \subseteq \R^d$ is open and of full measure.
  By definition, we have $\CalO_i = H^T C_i$ with $C_i \subseteq \CalO_i$ compact.
  Hence, setting $C := C_1 \cup C_2$, it follows that $C \subseteq \CalO$ is compact
  and satisfies $\CalO = H^T C$, showing defining condition \ref{enu:FrequencySupportCompactlyGenerated}. 
  Lastly, let $K \subseteq \CalO$ be compact.
  For each $x \in K$, there exist $i_x \in \{ 1,2 \}$ and a compact set $K_x \subseteq \CalO_{i_x}$
  satisfying $x \in K_x^{\circ}$.
  By compactness of $K$, there exist $N \in \N$ and points $x_1,\dots,x_N \in K$ satisfying
  $K \subseteq \bigcup_{\ell=1}^N K_{x_\ell}^{\circ} \subseteq \bigcup_{\ell=1}^N K_{x_\ell}$.
  Since $[K_{x_\ell}] \subseteq H \times K_{x_\ell} \subseteq H \times \CalO_{i_{x_\ell}} \subseteq H \times \CalO$
  is compact for each $\ell$, it follows that $[K] \subseteq H \times \CalO$ is a closed set satisfying
  $[K] \subseteq \bigcup_{\ell=1}^N [K_{x_\ell}]$, so that $[K] \subseteq H \times \CalO$
  is compact.
\end{proof}

In addition to \Cref{lem:DualOrbitUnion}, we will use two additional lemmata
for proving the uniqueness of the frequency support.
For this, given an integrably admissible dilation group $H$ 
and a compact set $K \subseteq \R^d$, we define
\[
  M_{K, H} f : 
  H \to [0,\infty), \quad
  h \mapsto \sup_{\xi \in K}
              f(h^{T} \xi)
\]
for any continuous function $f : \R^d \to [0,\infty)$.
The function $M_{K,H} f$ is lower semicontinuous
(as a supremum of continuous functions),
and hence Borel measurable. 

\begin{lemma}\label{lem:SpecialIntegrability}
  Let $H \leq \GL(d, \mathbb{R})$ be integrably admissible with Haar measure $\mu_H$,
  and let $\CalO \subseteq \R^d$ be an essential frequency support associated to $H$.
  Let $K \subseteq \CalO$ be compact
  and let $f \in C_c(\CalO)$ with $f \geq 0$.
  Then
  \[
    \int_H
      M_{K, H} f (h)
    \, d \mu_H(h)
    < \infty.
    \]
  \end{lemma}

\begin{proof}
  Define the compact set $\widetilde{K} := K \cup \supp f \subseteq \CalO$.
  Note that if $M_{K,H} f(h) \neq 0$, then there exists $\xi \in K \subseteq \widetilde{K}$
  satisfying $h^T \xi \in \supp f \subseteq \widetilde{K}$,
  and hence $(h,\xi) \in [\widetilde{K}]$,
  which is compact by defining condition \ref{enu:FrequencySupportProperness}.
  Therefore, using the projection $\pi_1 : H \times \R^d \to H, \; (h,\xi) \mapsto h$,
  it follows that $\supp (M_{H,K} f) \subseteq K_0 := \pi_1([\widetilde{K}])$.
  Since $0 \leq M_{H,K} f \leq \| f \|_{\sup}$, this implies
  \(
    \int_H
      M_{H,K} f(h)
    \, d \mu_H (h)
    \leq \| f \|_{\sup} \cdot \mu_H (K_0)
    < \infty ,
  \)
  as claimed.
\end{proof}

\begin{lemma}\label{lem:NonIntegrability}
  Let $H \leq \GL(d, \R)$ be integrably admissible with Haar measure $\mu_H$,
  and let $\CalO = H^T C$ with a compact set $C \subseteq \CalO$ be an
  essential frequency support associated to $H$.
  Let $f : \R^d \to [0,\infty)$ be continuous.
  Suppose that there exists $\xi_0 \in \partial \CalO$ satisfying $f(\xi_0) \neq 0$.
  Then
  \[
    \int_H
      M_{C, H} f(h)
    \, d \mu_H (h)
    = \infty .
  \]
\end{lemma}

\begin{proof}
  Since $\xi_0 \in \partial \CalO$, there exists a sequence
  $(\xi_n)_{n \in \N}$ in  $\CalO = H^T C$ satisfying $\xi_n \to \xi_0$ as $n \to \infty$. 
  Write $\xi_n = h_n^T c_n$ for some $h_n \in H$ and $c_n \in C$.
  Let $Q \subseteq H$ be a compact unit neighborhood.
  The remainder of the proof is split into two steps.
  \\~\\
  \textbf{Step 1.} In this step, we show that
  for each $N \in \N$ there exists $M = M(N) \in \N$ such that
  \[
    h_M Q \cap \bigcup_{\ell=1}^N h_\ell Q = \emptyset.
  \]
  Arguing by contradiction, assume there exists $N \in \N$ such that
  $h_M Q \cap \bigcup_{\ell=1}^N h_\ell Q \neq \emptyset$
  and hence $h_M \in \bigcup_{\ell=1}^N h_\ell Q Q^{-1} =: Q'$
  for every $M \in \N$.
  Then $\big( (h_M, c_M) \big)_{M \in \N}$ is a sequence in the compact set $Q' \times C$,
  so that there exists a subsequence $(M_k)_{k \in \N}$ and $(h,c) \in Q' \times C$
  satisfying $(h_{M_k}, c_{M_k}) \to (h,c)$ as $k \to \infty$.
  This implies
  \[
    \xi_0
    = \lim_{k \to \infty} \xi_{M_k}
    = \lim_{k \to \infty} h_{M_k}^T c_{M_k}
    = h^T c \in H^T C
    = \CalO
    = \CalO^{\circ},
  \]
  which contradicts the assumption $\xi_0 \in \partial \CalO$.
\\~\\
\textbf{Step 2.} By use of Step 1, we easily obtain a subsequence $(h_{n_\ell})_{\ell \in \N}$
  such that $(h_{n_\ell} Q)_{\ell \in \N}$ is pairwise disjoint.
  By definition of $M_{C, H} f$, it follows that
  \begin{align*}
    \int_{H}
      M_{C, H} f(h)
    \, d \mu_H(h)
    & \geq \sum_{\ell=1}^\infty
             \int_{h_{n_\ell} Q}
               M_{C, H} f(h)
             \, d \mu_H(h)
      \geq \sum_{\ell=1}^\infty
             \int_{h_{n_\ell} Q}
               f(h^T c_{n_\ell})
             \, d \mu_H(h) \\
    & =    \sum_{\ell=1}^\infty
             \int_{Q}
               f(q^T h_{n_\ell}^T c_{n_\ell})
             \, d \mu_H(q)
      =    \sum_{\ell=1}^\infty
             \int_{Q}
               f(q^T \xi_{n_\ell})
             \, d \mu_H(q).
  \end{align*}
  Since $\eps := f(\xi_0) > 0$ and $f$ is continuous, there exists an open set
  $Q_0 \subseteq Q$ and some $\delta > 0$ satisfying $f(q^T \xi) \geq \frac{\eps}{2}$
  for all $q \in Q_0$ and all $\xi \in \R^d$ with $\| \xi - \xi_0 \| \leq \delta$.
  In addition, since $\xi_{n_\ell} \to \xi_0$, there exists $\ell_0 \in \N$ satisfying
  $\| \xi_{n_\ell} - \xi_0 \| \leq \delta$ for all $\ell \geq \ell_0$.
  Overall, this implies that
  \[
    \int_H
      M_{C, H} f(h)
    \, d \mu_H(h)
    \geq \sum_{\ell=\ell_0}^\infty
           \int_{Q_0}
             f(q^T \xi_{n_\ell})
           \, d \mu_H(q)
    \geq \sum_{\ell=\ell_0}^\infty
           \mu_H(Q_0) \frac{\eps}{2}
    =    \infty ,
  \]
  which finishes the proof.
\end{proof}

For irreducibly admissible dilation groups (see \Cref{ex:dilations}),
a special case of \Cref{lem:NonIntegrability} (with $C$ being a singleton)
was obtained by A.\ Burtscheidt and the third named author;
see Proposition 5.4.1 in the thesis \cite{AchimPhD}. 

A combination of the preceding lemmas gives the desired uniqueness result.

\begin{theorem}\label{thm:DualOrbitUniqueness}
  Let $H \leq \GL(d, \R)$ be integrably admissible.
  Then the essential frequency support $\CalO \subseteq \R^d$ associated to $H$ is unique.
\end{theorem}

\begin{proof}
  By \Cref{lem:DualOrbitUnion}, it suffices to show that there cannot
  be two frequency supports $\CalO_1,\CalO_2 \subseteq \R^d$
  satisfying $\CalO_1 \subsetneq \CalO_2$.
  
  Arguing towards a contradiction, assume that such frequency supports $\CalO_1$
  and $\CalO_2$ do exist.
  Write $\CalO_i = H^T C_i$ with $C_i \subseteq \CalO_i$ compact
  and choose $\xi_0 \in \CalO_2 \setminus \CalO_1$.
  Fix $f \in C_c (\CalO_2)$ satisfying $f \geq 0$ and $f(\xi_0) > 0$.
  Since $\CalO_1 \subseteq \R^d$ is open and of full measure with $\xi_0 \notin \CalO_1$,
  it follows that $\xi_0 \in \partial \CalO_1$.
  Hence, an application of \Cref{lem:NonIntegrability} implies
  \[
    \int_H
      M_{C_1, H} f (h)
    \, d \mu_H(h)
    = \infty .
  \]
  On the other hand, since $C_1 \subseteq \CalO_1 \subsetneq \CalO_2$ is compact,
  it follows by \Cref{lem:SpecialIntegrability}
  that
   \[
     \int_H
       M_{C_1, H} f (h)
     \, d \mu_H(h)
     < \infty ,
   \]
   which is the required contradiction.
\end{proof}

\section{Quasi-isometry between \texorpdfstring{$H \times C$}{H x C} and the essential frequency support}
\label{sec:ConnectivityRespectingGroups}

This section is devoted to establishing a quasi-isometry between the product $H \times C$ and $\CalO$,
where $H$ is an integrably admissible dilation group satisfying some mild additional conditions,
and $\CalO = H^T C$ its essential frequency support, where we assume that $C \subseteq \CalO$
is connected and compact (see \Cref{def:connectivity}).

More precisely, we will show that the \emph{full orbit map}
\begin{align}\label{eq:full_orbit_map}
  p : 
  H \times C \to \CalO, \quad
  (h, \xi) \mapsto h^{-T} \xi,
\end{align}
forms a quasi-isometry relative to suitable metrics on $H \times C$ and $\CalO$.
This property will be used in \Cref{sec:coorbit_equivalence}
to characterize when two dilation groups yield the same coorbit space.

The first few subsections of this section are concerned with the construction of relevant objects
and recalling the requisite background.
The fact that the full orbit map $p$ is a quasi-isometry
from $H \times C$ into the essential frequency support $\CalO$ of $H$
is proven in \Cref{sec:coarseequivalence_proof}.

\subsection{Quasi-isometries}
\label{sec:quasi-isometry}

This subsection reviews the notions of a quasi-isometry and a quasi-inverse.
As we need these notions and results in various contexts,
we introduce them in the general setting of metric spaces.
Standard references%
\footnote{The terminology is not uniform throughout the literature,
  see, e.g., \cite[Remark 3.A.4]{cornulier2016metric} for some overview.}
on coarse geometry are, e.g., \cite{roe2003lectures, nowak2012large, cornulier2016metric}, and we refer to these references for further details.

Throughout this subsection, we let $(X, d_X)$ and $(Y, d_Y)$ be arbitrary metric spaces. 
A map $f : X \to Y$ is said to be a \emph{quasi-isometry}
if it satisfies the following two conditions:
\begin{enumerate}[label=(q\arabic*)]
  \item \label{enu:QuasiIsometryCondition1}
        There exist constants $R_1 , R_2>0$ such that
        \[
          R_1^{-1} d_X (x,x') - R_2
          \leq d_Y \big(f(x), f(x') \big)
          \leq R_1 d_X (x, x') + R_2
        \]
        for all $x, x' \in X$;

  \item \label{enu:QuasiIsometryCondition2}
        There exists $R_3 > 0$ such that, for every $y \in Y$,
        there exists $x \in X$, such that $d_Y \big( f(x), y\big) \leq R_3$.
\end{enumerate}
The composition of two quasi-isometries is again a quasi-isometry.
Hence the existence of a quasi-isometry between $X $ and $Y$ defines a transitive relation
between metric spaces (and this relation is clearly also reflexive).
This relation is also symmetric, i.e., for every quasi-isometry from $X$ into $Y$
there exists an associated quasi-isometry in the converse direction.
This observation will be useful for us in the following,
and is therefore spelled out in somewhat more detail in the following paragraph. 

A map $f_2: Y \to X$ is called a \emph{quasi-inverse} of a map $f_1 : X \to Y$ if
\[
  \sup_{y \in Y}
    d_Y\bigl(f_1(f_2(y)), y\bigr)
  < \infty ~. 
\]
Every quasi-isometry $f_1$ has a quasi-inverse:
By assumption \ref{enu:QuasiIsometryCondition2},
the axiom of choice provides a map $f_2 : Y \to X$ satisfying
\[
  \sup_{y \in Y}
    d_Y \bigl(f_1(f_2(y)), y\bigr)
  \leq R_3
  .
\]

The following result is folklore.
For reasons of self-containment, we provide a proof in \Cref{sec:postponedproofs}. 

\begin{lemma} \label{lem:inv_coarse_isom} Let $(X, d_X)$ and $(Y, d_Y)$ be metric spaces.
\begin{enumerate}[label=(\roman*)]
\item Any quasi-inverse of a quasi-isometry $f_1:X \to Y$ is a quasi-isometry. 

\item Let $f_1 : X \to Y$ denote a quasi-isometry, and $f_2$ a quasi-inverse of $f_1$.
      Then $f_1$ is a quasi-inverse of $f_2$. 
\end{enumerate}
\end{lemma}

\subsection{Connectivity-respecting dilation groups}

Let $H \leq \mathrm{GL}(d, \R)$ be an integrably admissible dilation group with essential frequency support $\CalO = H^T C$.
In order to show the quasi-isometry property of the full orbit map $p : H \times C \to \CalO$ defined in \Cref{eq:full_orbit_map},
we will construct adequate metrics on $H \times C$ and $\CalO$ in the following subsections.
For these constructions, we will impose some mild additional assumptions on the group $H$ and its frequency support $\CalO$ in \Cref{def:connectivity}.

In order to prepare for \Cref{def:connectivity}, we make the observation that the dual action of $H$ on $\mathcal{O}$
induces a permutation action on the connected components of $\mathcal{O}$.
More precisely, let
\[
  \Gamma
  := \{ \CalO_0 \subseteq \CalO \,\,\colon\,\, \CalO_0 \text{ is a connected component of } \CalO \}
  .
\]
Then for $\CalO_0 \in \Gamma$ and $h \in H$, we see that $h^T \CalO_0 \subseteq \CalO$ is connected
and hence $h^T \CalO_0 \subseteq \CalO_0'$ for some $\CalO_0' \in \Gamma$.
Then $h^{-T} \CalO_0'$ is connected with $\CalO_0 \subseteq h^{-T} \CalO_0'$ and thus $\CalO_0 = h^{-T} \CalO_0'$,
meaning $h^{T} \CalO_0 = \CalO_0' \in \Gamma$, so that the map
\begin{equation}
  H \times \Gamma \to \Gamma, \quad
  (h, \CalO_0) \mapsto h^{-T} \CalO_0
  \label{eq:ActionOnComponents}
\end{equation}
is well-defined.
It is straightforward to check that this map is a group action.

\begin{lemma}\label{lem:H0AutomaticallyClosed}
  Let $H \leq \mathrm{GL}(d, \R)$ be an integrably admissible dilation group with essential frequency support $\CalO$.
  Then, given any connected component $\CalO_0$ of $\CalO$, the stabilizer $H_0$ of $\CalO_0$,
  \[
    H_0 := \{ h \in H \,\,\colon\,\, h^T \CalO_0 = \CalO_0 \} \subseteq H
  \]
   is a closed subgroup.
\end{lemma}

\begin{proof}
 Let $(h_n)_{n \in \N}$ be a sequence in $H_0$ with $h_n \to h$ for some $h \in H$.
  Fix $\xi \in \CalO_0$, and let $\CalO_0' := h^T \CalO_0$,
  and note that $\CalO_0'$ and $\CalO_0$ are both connected components of $\CalO$; 
  see the discussion around \Cref{eq:ActionOnComponents}.
  Then $\CalO_0'$ is a neighborhood of $h^T \xi$, and we have
  $h_n^T \xi \rightarrow h^T \xi$ as $n \to \infty$.
  As such, there exists $n_0 \in \N$ such that $h_n^T \xi \in \CalO_0'$ for all $n \geq n_0$,
  which then implies $h_n^T \xi \in h_n^T \CalO_0 \cap \CalO_0' = \CalO_0 \cap \CalO_0'$.
  Since the two components $\CalO_0$ and $\CalO_0'$ have a nontrivial intersection,
  it follows that $h^T \CalO_0 = \CalO_0' = \CalO_0$, and hence $h \in H_0$.
\end{proof}

The following definition introduces the additional assumptions
on an integrably admissible dilation group that will be needed in the remainder.

\begin{definition} \label{def:connectivity}
    An integrably admissible dilation group $H \leq \mathrm{GL}(d, \R)$ with essential frequency support $\CalO$ is called \emph{connectivity-respecting} if it satisfies the following additional assumptions:
    \begin{enumerate}
        \item[(c1)] There exists a compact, \emph{connected} set $C \subseteq \CalO$ satisfying $\CalO = H^T C$;
        
        \item[(c2)] The stabilizer group $H_0 := \{ h \in H : h^T \CalO_0 = \CalO_0 \}$
                    of the connected component $\CalO_0$ of $\CalO$ containing $C$ is compactly generated.
    \end{enumerate}
\end{definition}
In principle, condition (c2) is a nontrivial restriction, since there do exist closed matrix groups that are not compactly generated, see \cite[Appendix A]{Ross_cg}. However, we currently do not know of an example where such groups arise as stabilizer groups $H_0$ obtained in the sense of (c2) from an integrably admissible matrix groups.

The following simple observation will be used repeatedly.

\begin{lemma} \label{lem:stabilizer_compactlygenerated} 
    If $H \leq \mathrm{GL}(d, \R)$ is connectivity-respecting,
    then the stabilizer $H_0$ of each connected component $\CalO_0$ of $\CalO$ is compactly generated.
\end{lemma}

\begin{proof}
    By the discussion around \Cref{eq:ActionOnComponents}, the group $H$ acts on the set $\Gamma$ of connected components of $\mathcal{O}$ by  $(h, \CalO_0) \mapsto h^{-T} \CalO_0$. Clearly, the set $H_0$ is the stabilizer group of $\mathcal{O}_0 \in \Gamma$ with respect to this action. 
    
    Furthermore, denoting by $C \subseteq \mathcal{O}_0$ the connected sets from (c1) of Definition \ref{def:connectivity}, the fact that $\mathcal{O} = H^T C = H^T \mathcal{O}_0$ implies that the action of $H$ on $\Gamma$ is transitive.  Hence, all fixed groups are conjugate to $H_0$, and if $H_0$ is compactly generated, the same applies to the remaining fixed groups. 
\end{proof}

We next discuss the connectivity-respecting property for the classes of integrably admissible dilation groups mentioned in \Cref{ex:dilations}.

\begin{example}[Connectivity-respecting dilation groups] \label{ex:connectivity_respecting} Let $H \leq \mathrm{GL}(d, \mathbb{R})$ be closed.
\begin{enumerate} 
\item If $H$ is integrably admissible and \emph{connected},
        then $H$ is connectivity-respecting if and only if $\CalO$ is connected.

        For this, first note that $\CalO = H^T C$ is connected whenever  $H$ and $C$ are connected.
        For the converse implication, the assumption that $H$ is integrably admissible furnishes a compact set $C_0 \subseteq \CalO$ with $\CalO = H^T C_0$. 
        Then, since $\CalO$ is connected, \Cref{lem:ConnectedSuperSet} yields
        a compact connected set $C \subseteq \CalO$ with $C \supseteq C_0$,
        so that $\CalO = H^T C_0 \subseteq H^T C \subseteq \CalO$.
        Since any connected locally compact group is compactly generated, all conditions are satisfied for $H_0 := H$.

        This observation shows that in particular all integrably admissible one-parameter groups
        in dimension $d>1$ are connectivity-respecting. 
        The frequency support associated to any of these groups
        is $\CalO = \mathbb{R}^d \setminus \{ 0 \}$, cf. \cite[Proposition 6.3]{FuehrVelthoven}.

  \item Assume that $H$ is irreducibly admissible, i.e., that $\CalO$ is a single orbit. Then condition (c1) of \Cref{def:connectivity} can be guaranteed by any singleton set $C = \{ \xi \} \subseteq\CalO$. The other condition is also fulfilled: 
        Let $H_1$ denote the connected component of the identity $\mathrm{id} \in H$ of $H$.  Then it is clear that any orbit stabilizer $H_0$ in the sense of \Cref{def:connectivity} contains $H_1$. By \cite[Remark 4]{fuhr1998continuous}, $H_1$ has finite index in $H$, and therefore in $H_0$. Being connected, $H_1$ is generated by any compact symmetric neighborhood of unity $V \subseteq H_1$, and then $H_0$ is generated by $V$ together with finitely many coset representatives of $H_0/H_1$.  

        Hence the following results are applicable to all irreducibly admissible groups, without imposing further technical conditions. This stands in contrast to the precursor paper \cite{CoarseGeometryPaper}, which uses the standing assumption that the dual (single point) stabilizers 
        \[
          H_\xi = \{ h \in H : h^T \xi = \xi \}~,
          \quad
          \xi \in \CalO,
        \]
        are contained in $H_1$. 
\end{enumerate}
\end{example}

\subsection{Word metrics on dilation groups}
\label{sec:WordMetric}

Let $H \leq \mathrm{GL}(d, \R)$ be a closed subgroup.
For a nonempty symmetric set $W \subseteq H$, define the map
$d_W : H \times H \to \mathbb{N}_0 \cup \{ \infty \}$ by
\begin{align} \label{eq:wordmetric}
  d_W(x,y)
  = \begin{cases}
      \inf \{ m \in \mathbb{N} : x^{-1} y \in W^m \}, & \text{if } x \neq y; \\
      0                                               & \text{if } x = y,
    \end{cases}
\end{align}
where we set $\inf \emptyset = \infty$.
It is well-known that $d_W$ is a metric on $H$
(with the exception that it can attain the value $\infty$)
and that $d_W$ is left-invariant,
i.e., $d_W(x,y) = d_W(zx, zy)$ for all $x,y, z \in H$;
see, e.g., \cite[Lemma 4.2]{CoarseGeometryPaper}.

In the remainder, 
we will choose $W$ to be an open, precompact symmetric generating set of a subgroup $H_0 \leq H$.
The next (elementary) lemma shows that the resulting metrics for different choices of $W$
are coarsely equivalent, so that the precise choice is immaterial for our purposes.
The proof is deferred to \Cref{sec:postponedproofs}.

\begin{lemma}\label{lem:DifferentWsAreCoarselyEquivalent}
  Let $H$ be a locally compact group, let $H_0 \leq H$ be closed, and let
  $W,V \subseteq H_0$ be nonempty, open, precompact, symmetric generating sets for $H_0$.

  Then the identity map $h \mapsto h$ is a quasi-isometry from $(H,d_W)$ onto $ (H,d_V)$.
\end{lemma}

\subsection{Covers of the frequency support}

For showing that the full orbit map \eqref{eq:full_orbit_map} is a quasi-isometry,
we need to construct an adequate metric on the frequency support $\CalO$ of a dilation group $H$.
Following an early observation in \cite{feichtinger1985banach},
we will define such a metric by first constructing an adequate cover of $\CalO$.

We will construct covers of the essential frequency support that are induced
by the dual action of the dilation group.
General covers of this type were already constructed
in \cite[Section 4.1]{FuehrVelthoven}, but for the purpose of the present paper
it is essential to have covers consisting of open connected sets.
In order to show the existence of such covers,
we will use defining condition (c1) of \Cref{def:connectivity}.

For the construction of the induced cover $\CalQ$,
we need some basic properties of well-spread families, which we recall next.
A family $(h_i)_{i \in I}$ in $ H$ is called  \textit{uniformly discrete}
if there exists an open set $U \subseteq H$ containing the identity element
such that $h_i U \cap h_j U = \emptyset $ holds for all $i,j \in I$ with $i \not= j$.
It is called \textit{uniformly dense} or \emph{$V$-dense} if there exists a relatively compact set
$V \subseteq H$ such that $H = \bigcup_{i \in I} h_i V$.
A family is called \textit{well-spread} if it is the finite union
of uniformly discrete sets and is uniformly dense.

A fact that we will be using repeatedly is the following lemma; see \cite[Lemma 4.2]{FuehrVelthoven}.

\begin{lemma}\label{lem:uniform_intersection}
  Let $K_1, K_2 \subseteq \CalO$ be compact and let $(h_i)_{i \in I} \subseteq H$ be uniformly discrete.
  For $h \in H$, let
  \[
    I_h (K_1, K_2)
    := \{ i \in I : h^{-T} K_1 \cap h_i^{-T} K_2 \neq \emptyset \}.
  \]
  Then there exists $R>0$ such that $\# I_h (K_1, K_2) \leq R$ for all $h \in H$.
\end{lemma}

The following result shows that any connectivity-respecting dilation group
admits an adequate connected cover of its frequency support.
More precisely, the cover is connected, admissible, and structured,
in the sense of \cite[Definition~2.5]{VoigtlaenderEmbeddingsOfDecompositionSpaces}.

\begin{lemma}\label{lem:NicecoverExists}
  Let $H \leq \GL(d, \R)$ be connectivity-respecting and let $C \subseteq \CalO$
  be a connected, compact set such that $\CalO = H^T C$.

  There exists an open, connected, bounded set $Q \subseteq \CalO$
  with $\overline{Q} \subseteq \CalO$ and $Q \supseteq C$,
  an open set $P \subseteq Q$ with $\overline{P} \subseteq Q$,
  and a well-spread family $(h_i)_{i \in I} $ in $H$
  such that the family $\CalQ = (h_i^{-T} Q)_{i \in I}$ has the following properties:
  \begin{enumerate}[label=(\roman*)]
      \item $\CalO = \bigcup_{i \in I} h_i^{-T} P = \bigcup_{i \in I} h_i^{-T} Q$;

      \item $\sup_{i \in I} \# \{ j \in I : h_i^{-T}Q \cap h_j^{-T} Q \neq \emptyset \} < \infty$;

      \item \(
              \sup_{ (i,j) \in I \times I: h_i^{-T} Q \cap h_j^{-T} Q \not= \emptyset}
                \| (h_i^{-T})^{-1} h_j^{-T} \|
              < \infty
            \).
  \end{enumerate}
\end{lemma}

\begin{proof}
  Let $B_\eps (C) := \{ \eta \in \R^d \colon \mathrm{dist}(\eta, C) < \eps \}$,
  where $\eps > 0$ is chosen so small that $\overline{B_\eps (C)} \subseteq \CalO$,
  which is possible since $C \subseteq \CalO$ with $C$ compact and $\CalO$ open.
  \Cref{lem:NeighborhoodConnected} shows that $B_\eps (C)$ is connected.
  Let $U \subseteq H$ be an open, connected, precompact unit neighborhood,
  and let $V \subseteq U$ be an open unit neighborhood with $\overline{V} \subseteq U$.
  The sets
  \[
    Q
    := U^{-T} B_\eps (C)
    = \bigcup_{u \in U} u^{-T} B_\eps (C)
    \quad \text{and} \quad
    P
    := V^{-T} B_{\eps / 2} (C)
  \]
  are open as unions of open sets, and $Q$ is connected as the image of the connected set
  $U \times B_\eps (C)$ under the continuous map $(h,\xi) \mapsto h^{-T} \xi$.
  Furthermore, we have
  \[
    \overline{Q}
    \subseteq \overline{U}^{-T} \overline{\widetilde{C}}
    \subseteq H^T \CalO
    = H^T H^T C
    = H^T C
    = \CalO
    ,
  \]
  meaning that $\overline{Q} \subseteq \CalO$ is compact.
  Moreover, $\overline{P} \subseteq \overline{V}^{-T} \overline{B_{\eps/2} (C)} \subseteq U^{-T} B_\eps (C) = Q$.
  Lastly, let $(h_i)_{i \in I} $ be a well-spread family in $H$ that is $V$-dense,
  i.e., $H = \bigcup_{i \in I} h_i \, V$; see, 
  e.g., \cite[Lemma~3.3]{FuehrGroechenigSamplingTheorems} for the existence of such a sequence.

  We next verify that $\CalQ := (h_i^{-T} Q)_{i \in I}$ has the properties (i)--(iii).
  For (i), note that
  \[
    \CalO
    \supseteq \bigcup_{i \in I}
              h_i^{-T} Q
    \supseteq \bigcup_{i \in I}
              h_i^{-T} P
    = \bigcup_{i \in I}
        h_i^{-T} V^{-T} C^\ast
    = \left( \bigcup_{i \in I} h_i V \right)^{-T} C^\ast
    \supseteq H^{-T} C
    = \CalO
    .
  \]
  
  For properties (ii) and (iii), we consider the sets 
  \[
    I_i
    := \{ j \in I ~:~h_i^{-T}Q \cap h_j^{-T} Q \not= \emptyset \}, \quad i \in I. 
  \]
  An application of \Cref{lem:uniform_intersection} yields $R > 0$ such that
  \[
    \sup_{i \in I} \# I_i \leq R,
  \]
  which shows (ii).
  
  Lastly, the fact that $h_j^{-1} h_i \!\in\! ((\overline{Q}, \overline{Q}))$ whenever $j \!\in\! I_i$
  (equivalently $i \!\in\! I_j$), with $((\overline{Q}, \overline{Q})) \subseteq H$ being compact
  by \Cref{cor:basic_integrably}, gives because of
  $\| (h_i^{-T})^{-1} h_j^{-T} \| = \| h_i^T h_j^{-T} \| = \|h_j^{-1} h_i \|$ that
  \[ 
    \sup_{ (i,j) \in I \times I: h_i^{-T} Q \cap h_j^{-T} Q \not= \emptyset}
      \| (h_i^{-T})^{-1} h_j^{-T} \|
    < \infty.
  \]
  This completes the proof.
\end{proof}

A cover $(h_i^{-T} Q)_{i \in I}$ of $\CalO$ as constructed in \Cref{lem:NicecoverExists}
will be referred to as an \emph{induced cover}. 
A notion that is of crucial importance in the remainder of the paper is that
of the \emph{weak equivalence} of two such covers.
This is the condition that will actually be used to determine
whether two scales of coorbit spaces are distinct.
The following definition is \cite[Definition 3.3.1]{VoigtlaenderEmbeddingsOfDecompositionSpaces}.

\begin{definition}\label{def:weakequivalence}
  Two covers $(Q_i)_{i \in I}$ and $(P_j)_{j \in J}$ of $\CalO$
  are said to be \emph{weakly equivalent} if
  \begin{align}
    \sup_{i \in I} \# \{ j \in J : Q_i \cap P_j \neq \emptyset \}
    + \sup_{j \in J} \# \{ i \in I : Q_i \cap P_j \neq \emptyset \}
    < \infty.
  \end{align}
\end{definition}
In the next subsection, a characterization of weak equivalence
in terms of quasi-isometries will be given.

\subsection{Metrics induced by covers}
\label{sec:coverDistance}

Following \cite{feichtinger1985banach}, we define a metric on an open set $\CalO \subseteq \R^d$
by means of a cover $\CalQ = (Q_i)_{i \in I}$ of $\CalO$.

A sequence of sets $Q_{i_1}, ..., Q_{i_m} \in \CalQ$ is said to be a \emph{$\CalQ$-chain}
(of length $m$) from $x \in G$ to $y \in G$ if $x \in Q_{i_1}$,
$y \in Q_{i_m}$ and $Q_{i_k} \cap Q_{i_{k+1}} \neq \emptyset$ for all $k = 1, ..., m-1$.
We write $\CalQ_m(x,y)$ for the set of all $\CalQ$-chains of length $m$ from $x$ to $y$.
The \emph{$\CalQ$-chain distance} is the map
$d_{\CalQ} : \CalO \times \CalO \to \mathbb{N}_0 \cup \{\infty\}$ defined by
\[
  d_{\CalQ} (x,y) = 
  \begin{cases}
      \inf \{ m \in \N : \CalQ_m (x,y) \neq \emptyset \}, & \text{if } x \neq y, \\
      0,                                                  & \text{if } x = y;
  \end{cases}
\]
where we set $\inf \emptyset = \infty$.
The map $d_{\CalQ}$ defines a metric on $\CalO$
(with the exception that it can attain the value $\infty$).

The significance of a $\CalQ$-chain distance for our purposes is that it allows to characterize
the weak equivalence of two covers.
The following result is \cite[Theorem 3.22]{CoarseGeometryPaper}
(see also \cite[Proposition 2.7]{berge2022large}) stated for the special covers
considered in the present paper.

\begin{theorem}\label{thm:WeakEquivQuasiIso}
  Let $H_1, H_2$ be integrably admissible dilation groups with associated essential frequency
  supports $\CalO_1 = \CalO = \CalO_2$.
  Let $(h_i)_{i \in I} \subseteq H_1$ and $(g_j)_{j \in J} \subseteq H_2$ be well-spread,
  let $Q,P \subseteq \CalO$ be open, \emph{connected} sets
  with compact closures $\overline{Q},\overline{P} \subseteq \CalO$,
  and assume that the families $\CalQ = (h_i^{-T} Q)_{i \in I}$
  and $ \CalP = (g_j^{-T} P)_{j \in J}$ are covers of $\CalO$.
  Then the following statements are equivalent:
  \begin{enumerate}[label=(\roman*)]
     \item The covers $\mathcal{Q}$ and $\mathcal{P}$ are weakly equivalent.

     \item The map
           \(
             \mathrm{id}:
             (\mathcal{O}, d_{\mathcal{Q}}) \to (\mathcal{O}, d_{ \mathcal{P}}),\ 
             \xi \mapsto \xi 
           \)
           is a quasi-isometry.
  \end{enumerate}
\end{theorem}

It will often be useful for us to work with a conveniently chosen $\CalP$
instead of a fixed given cover $\CalQ$.
The following lemma shows that we can do this essentially without loss of generality.

\begin{lemma}\label{lem:coverChoiceDoesNotMatter}
  Let $Q,P \subseteq \CalO$ be open, \emph{connected} and bounded
  with $\overline{Q}, \overline{P} \subseteq \CalO$
  and let $(h_i)_{i \in I}$ and $(g_j)_{j \in J}$ be well-spread in $H$.
  Define $\CalQ = (h_i^{-T} Q)_{i \in I}$ and $\CalP = (g_j^{-T} P)_{j \in J}$
  and assume that $\CalQ, \CalP$ both cover $\CalO$.

  Then the identity map
  \[
    \mathrm{id} : 
    (\CalO, d_\CalQ) \to (\CalO, d_\CalP), \quad
    \xi  \mapsto \xi 
  \]
  is a quasi-isometry.
\end{lemma}

\begin{proof}
  By \Cref{thm:WeakEquivQuasiIso}, it is enough to show that the covers
  $\CalQ,\CalP$ are weakly equivalent.
  To see this, note that \Cref{lem:uniform_intersection} yields a constant $R < \infty$
  such that for each $j \in J$, we have
  \[
    \# \big\{ i \in I \,\,\colon\,\, h_i^{-T}Q \cap g_j^{-T} P \neq \emptyset \big\}
    \leq \# \big\{
              i \in I
              \,\,\colon\,\,
              g_j^{-T} \overline{P} \cap h_i^{-T} \overline{Q} \neq \emptyset
            \big\}
    = \# I_{g_j} (\overline{P}, \overline{Q})
    \leq R
    .
  \]
  By symmetry, this easily shows that $\CalQ, \CalP$ are weakly equivalent.
\end{proof}

\subsection{Orbit maps}
\label{sec:coarseequivalence_proof}

In this subsection, we will prove the main result of this section,
namely the quasi-isometry property of the full orbit $p$ defined in \Cref{eq:full_orbit_map}.
Throughout, we fix the following notation:

\begin{notation}\label{notation}
Let $H \leq \mathrm{GL}(d, \R)$ be a connectivity-respecting dilation group
with essential frequency support $\CalO$.
    \begin{enumerate}[label=(A\arabic*)]
    \item \label{enu:NotationFrequencySupport}
         The set $C \subseteq \CalO$ is compact, \emph{connected} such that $\CalO = H^T C$,
         and $W \subseteq H_0$ is a nonempty, open, precompact,
         symmetric generating set for the stabilizer
         \[
           H_0 = \{ h \in H \,\,:\,\, h^T \CalO_0 = \CalO_0 \}
         \]
         of the connected component $\CalO_0$ of $\CalO$ containing $C$.

    \item \label{enu:Notationcover}
          The set $Q \subseteq \CalO$ is an open, bounded, \emph{connected} set satisfying $C \subseteq Q$ and 
          $\overline{Q} \subseteq \CalO$, and $(h_i)_{i \in I}$ is a well-spread family in $H$
          for which $\CalQ = (h_i^{-T} Q)_{i \in I}$
          is an induced cover (cf.\ \Cref{lem:NicecoverExists}).

    \item The metric on $H \times C$ is
          \[
            d_{H \times C} \bigl( (h,\xi), (g,\eta) \bigr)
            := d_W (h, g) + d_C (\xi, \eta)
            ,
          \]
          where $d_W$ is the word metric on $H$ determined by the set $W$  and $d_C$ is a bounded metric on $C$
          (e.g., the usual Euclidean metric).

    \item The metric on $\CalO$ is the cover-induced metric $d_{\CalQ}$
          associated to the cover $\CalQ$.
\end{enumerate}
\end{notation}

\begin{remark}
In (A1), the assumption that $W \subseteq H_0$ is a generating set for the stabilizer $H_0$ of $H$ differs slightly from the assumption in \cite{CoarseGeometryPaper}, where $W$ is assumed to be a unit neighborhood for the connected component of $H$. This subtle difference is ultimately responsible for the fact that we can establish a quasi-isometry between dilation group and the set $\mathcal{O}$ for our setting (see Corollary \ref{cor:orbitmap_quasi} below), whereas the analogous statement in the setting of \cite{CoarseGeometryPaper} requires an additional, somewhat artificial assumption; see Corollary 4.9 therein.
\end{remark}

We start with the following proposition, which is a crucial ingredient in establishing
one of the estimates in the definition of a quasi-isometry.

\begin{proposition}\label{prop:ConnectedSetInverseImage}
  If $h_1, h_2 \in H$ are such that $h_1^{-T} \CalO_0 \cap h_2^{-T} \CalO_0 \neq \emptyset $, then $h_2^{-1} h_1 \in H_0$.
  In particular, we have $H_\xi := \{ h \in H \colon h^{-T} \xi = \xi \} \subseteq H_0$ for all $\xi \in \CalO_0$.
\end{proposition}

\begin{proof}
  Let $h_1, h_2 \in H$ with $h_1^{-T} \CalO_0 \cap h_2^{-T} \CalO_0 \neq \emptyset$,
  which implies that
  \[ \CalO_0 \cap h_1^T h_2^{-T} \CalO_0
    = \CalO_0 \cap (h_2^{-1} h_1)^{T} \CalO_0
  \]
  is nonempty.
  As noted in the discussion around \Cref{eq:ActionOnComponents},
  both $\CalO_0$ and $(h_2^{-1} h_1)^T \CalO_0$ are connected components of $\CalO$.
  Hence, $(h_2^{-1} h_1)^T \CalO_0 = \CalO_0$, so that $h_2^{-1} h_1 \in H_0$.

  For the remaining part, note that if $\xi \in \CalO_0$ and $h \in H_\xi$, then
  \[
    \xi
    \in h^{-T} \{ \xi \} \cap \mathrm{id}^{-T} \{ \xi \}
    \subseteq h^{-T} \CalO_0 \cap \mathrm{id}^{-T} \CalO_0
    ,
  \]
  so that the first part shows $h = \mathrm{id}^{-1} h \in H_0$.
  Since this holds for all $h \in H_\xi$, it follows that $H_\xi \subseteq H_0$, as required.
\end{proof}

Using \Cref{prop:ConnectedSetInverseImage}, we can now show the first of the two estimates
that are needed to prove that the full orbit map forms a quasi-isometry.

\begin{theorem}\label{thm:PXiIsQuasiExpansive}
 With notation as in \Cref{notation}, define
  \[
    p_\xi : 
    H \to \CalO, \quad
    h \mapsto h^{-T} \xi,
  \]
  for $\xi \in \CalO$.
  Then there exists $R \in \N$ such that
  \begin{align} \label{eq:estimate1}
    d_W (g,h)
    \leq R  \cdot d_{\CalQ} \bigl(p_\xi (g), p_\eta (h)\bigr) + R
  \end{align}
  for all $g, h \in H$ and $\xi, \eta \in Q$.
\end{theorem}

\begin{proof}
  As a preparation, we note that $Q \subseteq \overline{Q} \subseteq \CalO_0$.
  Indeed, the set $Q \subseteq \CalO$ being connected with $C \subseteq Q$,
  it follows that $Q \subseteq \CalO_0$, because $\CalO_0$ is the connected component of $\CalO$
  containing $C$.
  Since connected components are closed subsets of the ambient space,
  it follows that $\CalO_0$ is closed in $\CalO$.
  Since $\overline{Q} \subseteq \CalO$
  (cf.\ \Cref{notation}), this entails $\overline{Q} \subseteq \CalO_0$.

  We define a suitable $R \in \N$,
  and then show the estimate \eqref{eq:estimate1}.
  By \Cref{cor:basic_integrably}, the set
  \[
    ((\overline{Q}, \overline{Q}))
    = \bigl\{ h \in H \colon h^T \overline{Q} \cap \overline{Q} \neq \emptyset \bigr\}
  \]
  is compact. Since $H_0 \subseteq H$ is closed by \Cref{lem:H0AutomaticallyClosed}, it follows that also the set
  $ ((\overline{Q}, \overline{Q})) \cap H_0$ is compact.
  The set $W$ being symmetric with $H_0 = \langle W \rangle$, we have
  $ ((\overline{Q}, \overline{Q})) \cap H_0 \bigcup_{m=1}^{\infty} W^m$, and thus
  $ ((\overline{Q}, \overline{Q})) \cap H_0 \subseteq W^{R}$ for some $R \in \N$, because $K \subseteq H_0$ is compact
  and $W \subseteq H_0$ is open.

  For showing the estimate \eqref{eq:estimate1}, we fix $\xi,\eta \in Q$
  and $g,h \in H$, and distinguish three cases:
  \\~\\
  \emph{Case 1.}
  If $d_{\CalQ} \bigl(p_\xi (g), p_\eta(h)\bigr) = \infty$, then the desired estimate is trivial.
  \\~\\
  \emph{Case 2.}
  If $d_{\CalQ} \bigl(p_\xi (g), p_\eta(h)\bigr) \!=\! 0$,
  then $g^{-T} \xi = p_\xi (g) \!=\! p_\eta (h) = h^{-T} \eta$.
  Since $\xi,\eta \in Q \subseteq \CalO_0$, it holds that
  $\emptyset \neq g^{-T} Q \cap h^{-T} Q \subseteq g^{-T} \CalO_0 \cap h^{-T} \CalO_0$.
  On the one hand, by \Cref{prop:ConnectedSetInverseImage}, this implies that $g^{-1}h \in H_0$.
  On the other hand, this implies that
  \[
    \emptyset
    \neq \overline{Q} \cap h^T g^{-T} \overline{Q}
    = \overline{Q} \cap (g^{-1} h)^T \overline{Q}.
  \]
  In combination, this shows that $g^{-1} h \in H_0 \cap ( (\overline{Q}, \overline{Q}))  \subseteq W^{R}$, and therefore $d_W (g,h) \leq R$,
  as desired.
  \\~\\
  \emph{Case~3.} If $0 < d_{\CalQ} \bigl(p_\xi (g), p_\eta(h)\bigr) < \infty$,
  then we set $m := d_{\CalQ} \bigl(p_\xi (g), p_\eta(h)\bigr) \in \N$.
  By definition of the $\CalQ$-chain distance $d_{\CalQ}$,
  there exist indices $i_1,\dots,i_m \in I$ satisfying
  $p_\xi (g) \in h_{i_1}^{-T} Q$ and $p_\eta (h) \in h_{i_m}^{-T} Q$, and 
  \[
    h_{i_j}^{-T} Q \cap h_{i_{j+1}}^{-T} Q \neq \emptyset
  \]
  for $1 \leq j \leq m-1$.

  First, note that since $\xi,\eta \in Q$, the properties $ p_\xi (g) \in h_{i_1}^{-T} Q$
  and $p_\eta (h) \in h_{i_m}^{-T} Q$ imply that
  $g^{-T} \overline{Q} \cap h_{i_1}^{-T} \overline{Q} \neq \emptyset$ and
  $h^{-T} \overline{Q} \cap h_{i_m}^{-T} \overline{Q} \neq \emptyset$, respectively.
  Likewise, $h_{i_j}^{-T} \overline{Q} \cap h_{i_{j+1}}^{-T} \overline{Q} \neq \emptyset$
  for $1 \leq j \leq m-1$.
  Since $\overline{Q} \subseteq \CalO_0$,  an application of \Cref{prop:ConnectedSetInverseImage}
  therefore shows that
  \[
    g^{-1} h_{i_1}, h_{i_j}^{-1} h_{i_{j+1}}, h_{i_m}^{-1} h \in  H_0
    .
  \]

  Second, note that, for $x,y \in H$, we have
  $x^{-1} y \in ( (\overline{Q}, \overline{Q}))$  if and only if
  $x^{-T} \overline{Q} \cap y^{-T} \overline{Q} \neq \emptyset$.

  A combination of these observations  shows that
  \[
    g^{-1} h_{i_1}, h_{i_j}^{-1} h_{i_{j+1}}, h_{i_m}^{-1} h
    \in ((\overline{Q}, \overline{Q})) \cap H_0
    \subseteq W^{R}
  \]
  for all $1 \leq j \leq m -1$.
  Therefore,
  \begin{align*}
    d_W (g,h)
    &\leq d_W (g,h_{i_1}) + \sum_{j=1}^{m-1} d_W (h_{i_j}, h_{i_{j+1}}) + d_W (h_{i_m}, h) \\
    &\leq (m+1) R \\
    &=    R \cdot d_{\CalQ} \bigl(p_\xi (g), p_\xi(h)\bigr) + R,
  \end{align*}
  which completes the proof.
\end{proof}

The following theorem provides the converse estimate to \Cref{thm:PXiIsQuasiExpansive}.

\begin{theorem}\label{thm:estimate2}
   With notation as in \Cref{notation}, define
  \[
    p_\xi : 
    H \to \CalO, \quad
    h \mapsto h^{-T} \xi,
  \]
  for $\xi \in \CalO$.
  Then there exists $R > 0$ such that
  \begin{align} \label{eq:estimate2}
    d_{\CalQ} \bigl(p_\xi (g), p_\eta (h)\bigr) \leq R \cdot d_W(g,h) + R
  \end{align}
  for all $g, h \in H$ and $\xi,\eta \in Q$.
\end{theorem}

\begin{proof}
We split the proof into two steps.
\\~\\
\textbf{Step 1.} In this step, we show that there exists $R>0$ such that
\begin{align} \label{eq:claim_step1}
  d_{\CalQ} \bigl(p_{\xi} (g), p_{\eta} (h)\bigr)
  \leq R
\end{align}
for all $\xi,\eta \in Q$ and all $g,h \in H$ with $d_W (g,h) \leq 1$.
For this, we fix $\xi,\eta \in Q$ and $g,h \in H$ with $d_W (g,h) \leq 1$, so that $h \in g W$.

As in the proof of \Cref{thm:PXiIsQuasiExpansive},
the set $\overline{Q} \subseteq \CalO_0$ is compact.
Since $\overline{W} \subseteq H$ is compact with $W \subseteq H_0$, and since $H_0 \subseteq H$
is closed (cf.\ \Cref{lem:H0AutomaticallyClosed}), this implies that
\[
  K_1
  := \overline{W}^{-T} \overline{Q}
  \subseteq H_0^{-T} \CalO_0
  \subseteq H_0^{T} \CalO_0
  = \CalO_0
\]
is compact.
Thus, \Cref{lem:ConnectedSuperSet} yields a connected compact set $K_2 \subseteq \CalO_0$
satisfying $K_1 \subseteq K_2$.
Finally, set $K_3 := \overline{Q}$.
Defining
\[
  I_g(K_2, K_3)
  := \big\{ i \in I \,\,:\,\, g^{-T} K_2 \cap h_i^{-T} K_3 \neq \emptyset \big\}, 
\]
an application of \Cref{lem:uniform_intersection} yields a constant $R > 0$
(independent of $g,h,\xi,\eta$) such that
\[
  \# I_g(K_2, K_3) \leq R.
\]  

For proving the claim \eqref{eq:claim_step1}, we will show that $d_{\CalQ} (x,z) \leq R$ for all $x, z \in K := g^{-T} K_2$. Note that this indeed yields  \eqref{eq:claim_step1} as
$
  p_\xi (g)
     = g^{-T} \xi
     \in g^{-T} W^{-T} \xi
     \subseteq g^{-T} K_1
     \subseteq g^{-T} K_2
     = K
$
and, similarly, $p_\eta (h) \in K$.
Hence, $d_{\mathcal{Q}}(p_\xi(g), p_\eta(h)) \leq R $, as claimed.

For showing that $d_{\CalQ} (x,z) \leq R$ for all $x, z \in K$, 
let $x \in K$ be arbitrary and define $C_x$ to be the subset of $K$ consisting of those $z \in K$ for which there exists
$m \leq R$ and $(Q^{(\ell)})_{\ell = 1}^m \in \mathcal{Q}_m(x,z)$
with $Q^{(\ell)} \cap K \neq \emptyset$ for $1 \leq \ell \leq m$
and such that the $(Q^{(\ell)})_{\ell=1}^m$ are pairwise distinct.
We will show that $C_x = K$, which then yields $d_{\mathcal{Q}}(x,z)\leq R$ for all $z \in K$.

First, note that  $C_x$ is relatively open in $K$. Indeed, 
if $z \in C_x$, as established by a suitable $\mathcal{Q}$-chain $(Q^{(\ell)})_{\ell=1}^m$,
then $z' \in C_x$ for all $z' \in Q^{(m)} \cap K$, and $Q^{(m)}$ is a set of the form
$Q^{(m)} = h^{-T} Q$, which is open since $Q$ is.

Second, we show that $C_x$ is also relatively closed in $K$.
For this purpose, let $z_{1}$ be an element of the closure $\overline{C_{x}}$ of $C_x$ in $K$.
Since $\mathcal{Q} = (Q_i)_{i \in I}$ is a cover of $\mathcal{O} \supseteq K$, there is some
$i \in I$ with $z_{1} \in Q_i$.
The set $K \cap Q_i$ being a relatively open neighborhood of $z_{1}$ in $K$,
this implies the existence of some $z \in (K \cap Q_i) \cap C_{x}$.
Hence, there is a $\mathcal{Q}$-chain
$Q'_{1}, \ldots , Q'_{m}$ of length $m \leq R$ that connects $x$ and $z$
and such that $Q_{\ell}' \cap K \neq \emptyset$ for all $1 \leq \ell \leq m$
and such that the $(Q_\ell ')_{\ell=1}^m$ are pairwise distinct.
We will now distinguish two cases and show that in both cases we have $z_1 \in C_x$,
thereby showing that $C_x$ is relatively closed in $K$.
\\~\\
\emph{Case 1:} $Q'_{j} = Q_i$ for some $j \in \{1, \ldots, m\}$. In this case,
$Q'_{1}, \ldots , Q'_{j}$ is a $\mathcal{Q}$-chain of length $j\leq R$
consisting of pairwise distinct sets that connects $x$ and $z_{1}$
and such that $Q'_{t} \cap K \neq \emptyset$ for all $1 \leq t \leq j$.
Hence, $z_1 \in C_x$.
\\~\\
\emph{Case 2:} $Q'_{j} \neq Q_i$ for all $j \in \{1, \ldots, m\}$. In this case,
 the sequence $Q_1 ', \dots, Q_m ', Q_i$ is contained in $\CalQ = (h_i^{-T} Q)_{i \in I}$
and each of these sets has nonempty intersection with $K = g^{-T} K_2$.
By choice of $R$, this entails $m+1 \le R$.
Thus, the sequence $Q_1 ', \dots, Q_m ', Q_i$ is a $\CalQ$-chain of length at most $R$
consisting of pairwise distinct sets that all have nonempty intersection with $K$
and connecting $x$ and $z_1$.
Hence also in this case, $z_1 \in C_x$.

Since $K$ is connected, the fact that $C_x$ is both relatively closed and relatively open in $K$
(and nonempty, since $x \in C_x$) implies that $C_{x} = K$,
and hence completes Step 1.
\\~\\
\textbf{Step 2.}
Using Step 1, we now prove the general statement of the theorem. 
Let $g, h \in H$.
Note that the statement is immediate whenever $d_W(g,h) = \infty$.
Moreover, note that if $d_W(g,h) = 0$, then $d_W (g,h) \leq 1$ and hence
Step~1 implies that $d_{\CalQ}(p_\xi(g), p_\eta (h)) \leq R = R \cdot d_W(g,h) + R$. 

It remains to consider the case that $d_W(g,h) \in \N$.
Define $k := d_W(g,h)$, and write $g^{-1}h = \Pi_{i = 1}^k w_i$ for suitable $w_i \in W$.
Setting $h_0 := g$ and $h_j := g \cdot \Pi_{i = 1}^j w_i$ for $j=1,\ldots,k$,
it follows that $d_W(h_j,h_{j+1}) \le 1$ for all $0 \le j < k$.
An application of the triangle inequality for $d_\mathcal{Q}$ therefore yields 
\begin{align*}
  d_{\mathcal{Q}} \bigl(p_{\xi}(g),p_{\eta}(h)\bigr)
  & \leq d_\CalQ (p_\xi(g), p_\xi(h)) + d_\CalQ (p_\xi (h), p_\eta (h)) \\
  & \leq R + \sum_{j=0}^{k-1} d_{\mathcal{Q}} \bigl(p_\xi(h_j), p_{\xi} (h_{j+1})\bigr)
    \leq R + k R
    = R \cdot d_W(g,h) + R,
\end{align*} 
where the second inequality used inequality (\ref{eq:claim_step1}), observing that $R$ is independent of the choice of the pair $(\xi,\eta)$.
\end{proof}

We can now state and prove the main result of this section.

\begin{theorem}\label{thm:FullOrbitMapIsQuasiIsometry}
 With notation as in \Cref{notation}, the full orbit map
  \[
    p : 
    (H \times C, d_{H \times C}) \to (\CalO, d_{\CalQ}), \quad
    (h,\xi) \mapsto h^{-T} \xi,
  \]
  is a surjective quasi-isometry.
\end{theorem}

\begin{proof}
  By \Cref{thm:PXiIsQuasiExpansive,thm:estimate2}, it follows that, for
  $(h,\xi),(g,\eta) \in H \times C$,
  \begin{align*}
    d_{H \times C} ( (h,\xi), (g,\eta)) 
    &= d_{W} (h,g) + d_C (\xi,\eta) \\
    &\asymp d_W (h,g) + 1 \\ 
    &\asymp d_{\CalQ} (p_\xi (h), p_{\eta}(g)) + 1 \\
    &= d_{\CalQ} (p(h,\xi), p(g,\eta)) + 1,
  \end{align*}
  where the second step used that  $d_C$ is bounded.
  This implies that $p$ satisfies Condition~\ref{enu:QuasiIsometryCondition1}.

  Moreover, since $p(H \times C) = H^{-T} C = H^T C = \CalO$, the map  $p$ is surjective.
  This easily implies that Condition~\ref{enu:QuasiIsometryCondition2} is satisfied as well.
\end{proof}

The following consequence of \Cref{thm:FullOrbitMapIsQuasiIsometry} is what actually will be used in most of our applications.

\begin{corollary} \label{cor:orbitmap_quasi}
    With notation as in \Cref{notation}, the orbit map
    \[
      p_{\xi} : 
      (H, d_W) \to (\CalO, d_{\CalQ}), \quad
      \xi \mapsto h^{-T} \xi,
    \]
    is a quasi-isometry for each $\xi \in \CalO_0$.
\end{corollary}

\begin{proof}
First, let $\xi \in C$ and note that the inclusion map $\iota : H \to H \times C, \; h \mapsto (h, \xi)$ is a quasi-isometry.
Since the full orbit map $p : (H \times C, d_{H\times C}) \to (\CalO, d_{\CalQ})$ is a quasi-isometry
by \Cref{thm:FullOrbitMapIsQuasiIsometry}, it follows that $p_{\xi} = p \circ \iota$
is a quasi-isometry as the composition of quasi-isometries.
Since $\xi$ was arbitrary, this proves the claim for all $\xi \in C$.

Second, let $\xi \in \CalO_0$.
Using \Cref{lem:ConnectedSuperSet}, we can choose a compact, connected set
$C_{\xi} \subseteq \CalO$ satisfying $\{\xi \} \cup C \subseteq C_{\xi}$.
Since $\CalO = H^T C \subseteq H^T C_\xi \subseteq H^T \CalO = \CalO$, it follows that also $\CalO = H^T C_{\xi}$.
Moreover, $C_{\xi}$ is also contained in the connected component $\CalO_0$ of $\CalO$ containing $C$,
and as such the stabilizer subgroup of its connected component is compactly generated.
This shows that $C_{\xi}$ satisfies condition (A1) of \Cref{notation}.
In addition, using \Cref{lem:NicecoverExists}, there exists an open, bounded, connected set $Q_\xi$
satisfying $C \subseteq C_{\xi} \subseteq Q_{\xi}$ and $\overline{Q_{\xi}} \subseteq \mathcal{O}$
together with a well-spread family $(g_j)_{j \in J}$ in $H$
such that $\CalQ' = (g_j Q_{\xi})_{j \in J}$ is an induced cover of $\CalO$.
This shows that $C_{\xi}$ also satisfies the remaining conditions of \Cref{notation}.
Hence, by the previous paragraph, it follows that $p_{\xi} : (H, d_W) \to (\CalO, d_{\CalQ'})$ is a quasi-isometry.
Since the identity between $(\CalO, d_{\CalQ})$ and $(\CalO, d_{\CalQ'})$
forms a quasi-isometry by \Cref{lem:coverChoiceDoesNotMatter},
this implies the claim for all $\xi \in \CalO_0$. 
\end{proof}

\section{Coorbit spaces associated to different dilation groups}
\label{sec:coorbit_equivalence}

This section exploits the coarse geometric results obtained in the previous sections
to compare coorbit spaces associated with different dilation groups. We start by recalling the central notions on wavelet coorbit spaces.

\subsection{Wavelet transforms} 

Let $H \leq \mathrm{GL}(d, \R)$ be a closed subgroup with Haar measure $\mu_H$
and modular function $\Delta_H$.
The semidirect product group $G = \mathbb{R}^d \rtimes H$ is the set $\mathbb{R}^d \times H$
equipped with the group law $(x_1, h_1) (x_2 ,h_2) = (x_1 + h_1 x_2, h_1 h_2)$.
A direct verification shows that a left Haar measure $\mu_G$ on $G$ is given by
$d \mu_G (x,h) = |\det h|^{-1} dx d\mu_H (h)$,
and that the modular function $\Delta_G$ on $G$ is given by
$\Delta_G (x,h) = |\det h|^{-1} \Delta_H (h)$.

The \emph{quasi-regular representation} of $G$ on $L^2 (\R^d)$ is given by
\[
  [\pi(x, h) f](t)
  = |\det h|^{-\frac{1}{2}} f(h^{-1} (t-x)),
  \quad t \in \R^d.
\]
For a nonzero $\psi \in L^2 (\R^d)$, its associated \emph{wavelet transform}
is the map $W_{\psi} : L^2 (\R^d) \to L^{\infty} (G) $ defined by
\[
  W_{\psi} f (x,h)
  = \langle f, \pi(x,h) \psi \rangle,
  \quad (x,h) \in \R^d \times H,
\]
for $f \in L^2 (\R^d)$.
A function $\psi \in L^2 (\R^d)$ is said to be \emph{admissible} if $W_{\psi}$ is an
 isometry from $L^2 (\R^d)$ into $L^2 (G)$.
Equivalently, a function $\psi \in L^2 (\R^d)$ is admissible
if and only if its Fourier transform $\widehat{\psi}$ satisfies
\[
  \int_H |\widehat{\psi} (h^T \xi)|^2 \; d\mu_H (h)
  = 1
\]
for a.e.\ $\xi \in \R^d$; see, e.g., \cite[Theorem 1]{fuehr2002continuous}
and \cite[Theorem 1.1]{laugesen2002characterization}.

The significance of integrably admissible dilation groups for wavelet and coorbit theory
is that they guarantee the existence of admissible vectors with convenient additional properties.
The following result is \cite[Proposition 2.7]{currey2016integrable}
(see also \cite[Theorem 2.10]{FuehrVelthoven}).

\begin{proposition}[\cite{currey2016integrable}]\label{prop:CcInftyAdmissibleVectorExists}
  Let $H \leq \mathrm{GL}(d, \R)$ be integrably admissible with frequency support $\CalO$.
  Then there exists an admissible vector $\psi \in L^2 (\R^d)$
  with Fourier transform $\widehat{\psi} \in C_c^{\infty} (\CalO)$.
\end{proposition}

\subsection{Coorbit spaces}
\label{sec:CoorbitSpaces}

For defining wavelet coorbit spaces defined by integrably admissible dilation groups,
we follow the concrete approach in \cite{FuehrVelthoven};
see also \cite{velthoven2022coorbit} for an abstract approach for general (possibly reducible) integrable group representations. 

Let $H \leq \mathrm{GL}(d, \R)$ be integrably admissible with essential frequency support $\CalO$.
Denoting by $\mathcal{F}^{-1}$ the inverse Fourier transform,
define the space $\Schwartz_{\CalO} := \mathcal{F}^{-1} (C_c^{\infty} (\CalO))$,
and equip it with the topology making
$\mathcal{F}^{-1} : C_c^{\infty} (\CalO) \to \mathcal{S}_{\CalO}$ into a homeomorphism,
with respect to the usual topology on $C_c^{\infty} (\CalO)$.
The anti-dual space of $\Schwartz_{\CalO}$, that is,
the space of all conjugate-linear continuous functionals on $\Schwartz_{\CalO}$,
is denoted by $\Schwartz_{\CalO}^*$, and equipped with the weak$^*$-topology.
We write $\langle \cdot, \cdot \rangle$ for the \emph{sesquilinear} pairing
between $\Schwartz^*_{\CalO}$ and $\Schwartz_{\CalO}$,
that is, $\langle f, \varphi \rangle := f(\varphi)$ for $f \in \Schwartz_{\CalO}^\ast$
and $\varphi \in \Schwartz_{\CalO}$.
To be consistent with \cite{FuehrVelthoven},
we define the Fourier transform $\widehat{f}$ of $f \in \Schwartz_{\CalO}^\ast$
as
\[
  \widehat{f} (\varphi ) := \langle f, \overline{\widehat{\varphi}\,} \, \rangle 
  \qquad \text{for } \varphi \in C_c^\infty(\CalO)
  .
\]
Note that $\widehat{f}$ forms a continuous linear functional on $C_c^{\infty} (\CalO)$ as
\(
  \Fourier \overline{\widehat{\varphi}\,}
  = \Fourier \Fourier^{-1} \overline{\varphi}
  = \overline{\varphi}
  \in C_c^\infty(\CalO),
\)
and thus $\overline{\widehat{\varphi}\,} \in \Schwartz_{\CalO}$ for $\varphi \in C_c^\infty (\CalO)$.

To define the coorbit spaces, fix an admissible vector $\psi \in \Schwartz_{\CalO}$ (see \Cref{prop:CcInftyAdmissibleVectorExists}).
Then also
$\pi(x,h) \psi \in \Schwartz_{\CalO}$ for $(x,h) \in \R^d \rtimes H$ (see \cite[Lemma 2.9]{FuehrVelthoven}). As such, we can define the (extended) \emph{wavelet transform} of $f \in \Schwartz^*_{\CalO}$
 as $W_{\psi} f = \langle f, \pi (\cdot) \psi \rangle$. 
We note that $W_\psi f : \R^d \rtimes H \to \CC$ is (Borel) measurable.

For $p, q \in [1, \infty]$, we define the coorbit space $\Co(L^{p,q} (G))$
as the space of all $f \in \Schwartz^*_{\CalO}$ satisfying
\[
  \| f \|_{\Co(L^{p,q}(G))}
  := \| W_\psi f \|_{L^{p,q}(G)}
  := \bigg(
       \int_H
         \big\| W_{\psi} f (\cdot, h) \big\|_{L^p}^q \;
       \frac{d\mu_H (h)}{|\det h|}
     \bigg)^{1/q}
  < \infty
\]
for $p \in [1, \infty]$, $q \in [1, \infty)$, and
\[
  \| f \|_{\Co(L^{p, \infty} (G))}
  := \| W_\psi f \|_{L^{p, \infty}(G)}
  := \esssup_{h \in H} \| W_{\psi} f (\cdot, h ) \big\|_{L^p}.
\]
The spaces $\Co(L^{p,q}(G))$ are  Banach spaces that are independent
of the chosen admissible defining vector $\psi \in \Schwartz_{\CalO}$,
cf.\ \cite[Proposition 3.3]{FuehrVelthoven}.
For $p=q=2$, we have $\Co(L^{p,q}(G)) = L^2 (\R^d)$,
up to canonical identifications,
which can be deduced from a combination of \cite[Proposition 2.19]{FuehrVelthoven}
and \cite[Lemma 4.13]{velthoven2022coorbit}.
In particular, by using \cite[Theorem 7.4]{velthoven2022coorbit}, this implies that
\begin{align}\label{eq:identification}
  \Co(L^{p,q}(G))
  = \big\{ f \in L^2 (\R^d) \,\,:\,\, \| W_{\psi} f \|_{L^{p,q}(G)} < \infty \big\},
\end{align}
up to canonical identifications, for $1 \leq p, q \leq 2$.
See \cite[Remark 2.11]{CoarseGeometryPaper} for further details.

The coorbit space $\Co(L^{p,q}(G))$ can alternatively be described by a Besov-type norm.
To be more explicit, by \cite[Theorem 5.5]{FuehrVelthoven},
the coorbit space norm $\| \cdot \|_{\Co(L^{p,q}(G))}$ is equivalent to the Besov-type norm
\[
  \| f \|_{\mathcal{D}(\CalQ, L^p, \ell^q_u)}
  := \bigg\|
       \Big(
         u_i
         \cdot \big\|
                 \mathcal{F}^{-1} \bigl(\varphi_i \cdot \widehat{f} \,\,\bigr)
               \big\|_{L^p}
       \Big)_{i \in I}
     \bigg\|_{\ell^q}
  \qquad \text{for } f \in \Schwartz_{\CalO}^\ast ,
\]
where $\CalQ = (h_i^{-T} Q)_{i \in I}$ is an induced cover of $\CalO$,
$(\varphi_i)_{i \in I}$ is an adequate associated partition of unity,
and $u_i = |\det(h_i)|^{\frac{1}{2} - \frac{1}{p}}$ for $i \in I$.

\subsection{Equivalence of dilation groups}

The following definition is the central notion for the comparison of coorbit spaces. This definition is a natural extension of \cite[Definition 2.16]{CoarseGeometryPaper}.

\begin{definition}\label{defn:coorbit_equivalent}
  Let $H_{1}, H_{2} \le \GL(d, \mathbb{R})$ denote integrably admissible matrix groups.
  We call $H_{1},H_{2}$ \textit{coorbit equivalent} if, for all
  $1 \leq p,q \le \infty $ and for all $f \in L^2  (\mathbb{R}^{d})$, the norm equivalence
  \begin{equation*}
    \| f \|_{ \Co(L^{p,q}(\mathbb{R}^{d} \rtimes H_{1}))}
    \asymp \| f \|_{\Co(L^{p,q}(\mathbb{R}^{d} \rtimes H_{2}))}
  \end{equation*}
  holds.
  Here the norm equivalence is understood in the generalized sense that one
  side is infinite if and only if the other side is.
\end{definition}

We then immediately get the following analog of results from \cite[Theorem 2.18]{CoarseGeometryPaper}.
Note that the condition $\mathcal{O}_1 = \mathcal{O}_2$ emphasizes the importance
of Theorem \ref{thm:DualOrbitUniqueness},
stating that there is at most one essential frequency support for a given dilation group. 

\begin{theorem}\label{thm:coorbit_equiv_dual_orbits}
  Let $H_{1}, H_{2} \le \GL(d, \mathbb{R})$ denote integrably admissible matrix groups,
  and let $\mathcal{O}_{1},\mathcal{O}_{2}$ denote the associated essential frequency supports.
  Then the following are equivalent:
  \begin{enumerate}[label=(\roman*)]
  \item $H_{1}$ and $H_{2}$ are coorbit equivalent in the sense of \Cref{defn:coorbit_equivalent}.

  \item $Co(L^{p,q}(\mathbb{R}^{d} \rtimes H_{1})) = Co(L^{p,q}(\mathbb{R}^{d} \rtimes H_{2}))$
        as subspaces of $L^{2}(\mathbb{R}^{d})$, for all $1 \le p,q \le 2$;

  \item $Co(L^{p,q}(\mathbb{R}^{d} \rtimes H_{1})) = Co(L^{p,q}(\mathbb{R}^{d} \rtimes H_{2}))$
        as subspaces of $L^{2}(\mathbb{R}^{d})$, for some $1 \le p,q \le 2$ with $(p,q) \neq (2,2)$;

  \item $\mathcal{O}_{1} = \mathcal{O}_{2}$, and the covers induced
        by $H_{1}$ and $H_{2}$ on the common essential frequency support are weakly equivalent.
  \end{enumerate}
\end{theorem}

\begin{proof}
The proof is essentially the same as that of \cite[Theorem 2.18]{CoarseGeometryPaper},
and hence we only sketch it and provide the relevant references
for integrably admissible dilation groups.

The fact that (i) implies (ii) follows directly by the identification in \Cref{eq:identification}.
That (ii) implies (iii) is trivial.

Assume assertion (iii) holds.
Then, by \cite[Theorem 5.5]{FuehrVelthoven}, also the Besov-type decomposition spaces
$\mathcal{D}(\CalQ, L^p, \ell^q_u)$ and $\mathcal{D}(\CalP, L^p, \ell^q_{u'})$
associated to the two covers $\CalQ = (h_i^{-T} Q)_{i \in I}$
and $\CalP = (g_j^{-T} P)_{j \in J}$ induced by, respectively, $H_1$ and $H_2$,
and with weights $u_i = |\det (h_i)|^{\frac{1}{2} - \frac{1}{p}}$
and $u_j ' = |\det (g_j)|^{\frac{1}{2} - \frac{1}{p}}$, coincide.
In particular, this implies that the two norms associated to the decomposition spaces
are equivalent on $C_c^\infty (\CalO_1 \cap \CalO_2)$.

Since $\CalO_1, \CalO_2$ are both of full measure and hence dense in $\R^d$, it follows that also $\CalO_1 \cap \CalO_2 \subseteq \R^d$
is dense, and hence unbounded.
Now, if the condition $\CalO_1 \cap \partial \CalO_2 \neq \emptyset$
or the condition $\CalO_2 \cap \partial \CalO_1 \neq \emptyset$ was satisfied,
then, since we know from above that the norms on the two
decomposition spaces from above are equivalent,
\cite[Theorem 6.9 $\frac{1}{2}$]{VoigtlaenderEmbeddingsOfDecompositionSpaces}
would imply $p = q = 2$ and thus provide a contradiction.
Hence, we get that $\CalO_1 \cap \partial \CalO_2 = \emptyset$
and $\CalO_2 \cap \partial \CalO_1 = \emptyset$.
Note that $\partial \CalO_i = \R^d \setminus \CalO_i$, since $\CalO_i \subseteq \R^d$
is open and dense.
Therefore, we conclude that $\CalO_1 \subseteq \CalO_2$ and $\CalO_2 \subseteq \CalO_1$
and hence $\CalO_1 = \CalO_2$, as claimed.
The fact that the covers $\CalQ$ and $\CalP$ are weakly equivalent
follows from \cite[Theorem 6.9]{VoigtlaenderEmbeddingsOfDecompositionSpaces}
since $\mathcal{D}(\CalQ, L^p, \ell^q_u) = \mathcal{D}(\CalP, L^p, \ell^q_{u'})$
for some $(p,q) \neq (2,2)$.
In combination, this shows that (iv) holds.

Lastly, assume that $\mathcal{O}_{1} = \mathcal{O}_{2}$ and that the covers
$\CalQ = (h_i^{-T} Q)_{i \in I}$ and $\CalP = (g_j^{-T} P)_{j \in J}$ induced
by, respectively, $H_{1}$ and $H_{2}$ are weakly equivalent.
Then an application of \cite[Lemma 2.8]{CoarseGeometryPaper}
implies that
$|\det (h_i) |^{-1} \asymp |\det (g_j) |^{-1} $
for all $i \in I$ and $j \in J$ satisfying $h_i^{-T} Q \cap g_j^{-T} P \neq \emptyset$.
In turn, this implies that the weights
$u_i := |\det (h_i)|^{\frac{1}{2} - \frac{1}{p}}$ and $u_j ' := |\det (g_j)|^{\frac{1}{2} - \frac{1}{p}}$
satisfy $ u_i \asymp u_j'$ for all $i \in I$ and $j \in J$
satisfying $h_i^{-T} Q \cap g_j^{-T} P \neq \emptyset$. 
Together with the fact that $\CalQ$ and $\CalP$ are weakly equivalent,
this shows that the hypotheses of \cite[Lemma 6.11]{VoigtlaenderEmbeddingsOfDecompositionSpaces}
are satisfied, which then implies that 
$\mathcal{D}(\CalQ, L^p, \ell^q_u) = \mathcal{D}(\CalP, L^p, \ell^q_{u'})$
for all $1 \le p,q \leq \infty$.
Hence, by \cite[Theorem 5.5]{FuehrVelthoven}, it follows that
$\| \cdot \|_{\Co(L^{p,q} (\R^d \rtimes H_1))} \asymp \| \cdot \|_{\Co(L^{p,q} (\R^d \rtimes H_2))}$.
\end{proof}

We can now formulate our general criterion for coorbit equivalence of dilation groups. The next theorem is the main result of this section.

\begin{theorem}\label{thm:char_coorbit_equiv_new}
 Let $H_1, H_2 \leq \mathrm{GL}(d, \R)$ be connectivity-respecting dilation groups
 with essential frequency supports $\CalO_1 = H_1^T C_1$ and  $\CalO_2 = H^T_2 C_2$.
 With notation as in \Cref{notation},  let
  \[
    p^{(1)} : 
    (H_1 \times C_1, d_{H_1 \times C_1}) \to (\CalO_1, d_{\CalQ}), \quad
    (h, \xi) \mapsto h^{-T} \xi
  \]
  and
  \[
    p^{(2)} : 
    (H_2 \times C_2, d_{H_2 \times C_2}) \to (\CalO_2, d_{\CalP}), \quad
    (g, \eta) \mapsto g^{-T} \eta
    ,
  \]
  be full orbit maps and let $p^{(2)}_* : \CalO_2 \to H_2 \times C_2$ be a right inverse for $p^{(2)}$.
  
  Then the following assertions are equivalent:
  \begin{enumerate}[label=(\roman*)]
      \item $H_1$ and $H_2$ are coorbit equivalent;

      \item $\CalO := \mathcal{O}_1 = \mathcal{O}_2$, and the map
                    \[
                      p^{(2)}_* \circ p^{(1)} : 
                      (H_1 \times C_1, d_{H_1 \times C_1})
                      \to (H_2 \times C_2, d_{H_2 \times C_2}) 
                    \]
                    is a quasi-isometry.
  \end{enumerate}
\end{theorem}

\begin{proof}
  Suppose that (i) holds. 
  Since $H_1$ and  $H_2$ are coorbit equivalent, an application of \Cref{thm:coorbit_equiv_dual_orbits}
  shows that we have $\CalO_1 = \CalO_2 =: \CalO$ and that the covers
  $\CalQ$ associated to $H_1$ and $\CalP$ associated to $H_2$
  (chosen as in \Cref{notation})
  are weakly equivalent.
  Therefore, \Cref{thm:WeakEquivQuasiIso} shows that the identity map
  $\mathrm{id}_{\CalO} : (\CalO, d_{\CalQ}) \to (\CalO, d_{\CalP})$
  is a quasi-isometry.
  Next, \Cref{thm:FullOrbitMapIsQuasiIsometry} shows that the full orbit maps
  $p^{(1)}$ and $p^{(2)}$ from the statement of the current theorem are quasi-isometries.
  Since $p^{(2)}_\ast$ is a right inverse to $p^{(2)}$, it is easy to see that it is also a
  quasi-inverse (see \Cref{sec:quasi-isometry}),
  and thus \Cref{lem:inv_coarse_isom} shows that
  $p^{(2)}_\ast : (\CalO, d_{\CalP}) \to (H_2 \times C_2, d_{H_2\times C_2})$
  is a quasi-isometry as well.
  Since compositions of quasi-isometries are again quasi-isometries,
  this finally implies that
  \[
   p^{(2)}_\ast \circ p^{(1)} = p^{(2)}_\ast \circ \mathrm{id}_{\CalO} \circ p^{(1)}
  \]
  is a quasi-isometry, as required.

  \medskip{}

  Conversely, suppose that (ii) holds.
  \Cref{thm:FullOrbitMapIsQuasiIsometry} shows that the full orbit maps
  $p^{(1)}$ and $p^{(2)}$ from the statement of the current theorem are quasi-isometries.
  Since $p^{(2)}_\ast$ is a right inverse to $p^{(2)}$, it is easy to see that it is also a
  quasi-inverse (see \Cref{sec:quasi-isometry}),
  and thus \Cref{lem:inv_coarse_isom} shows that
  $p^{(2)}_\ast : (\CalO, d_{\CalP}) \to (H_2 \times C_2, d_{H_2} \times C_2)$
  is a quasi-isometry as well.

  Similarly, letting $p^{(1)}_\ast : \CalO \to H_1 \times C_1$ be any right inverse to $p^{(1)}$,
  we also see that $p^{(1)}_\ast : (\CalO,d_{\CalQ}) \to (H_1 \times C_1, d_{H_1 \times C_1})$
  is a quasi-isometry.

  Since compositions of quasi-isometries are again quasi-isometries, we thus see that
  the identity map $\mathrm{id}_{\CalO} : (\CalO, d_\CalQ) \to (\CalO, d_{\CalP})$
  is a quasi-isometry, since
  \[
    \mathrm{id}_{\CalO}
    = p^{(2)} \circ (p^{(2)}_{\ast} \circ p^{(1)}) \circ p^{(1)}_{\ast}
    .
  \]
  Hence, \Cref{thm:WeakEquivQuasiIso} shows that the covers $\CalQ,\CalP$ are weakly equivalent,
  so that \Cref{thm:coorbit_equiv_dual_orbits} implies that $H_1$ and $H_2$
  are coorbit equivalent.
\end{proof}

The following necessary condition of coorbit equivalence is what actually will be used in most of our applications.
Its proof is similar to the first part of the proof of \Cref{thm:char_coorbit_equiv_new}
(using \Cref{cor:orbitmap_quasi} instead of \Cref{thm:FullOrbitMapIsQuasiIsometry}), and hence we skip it.

\begin{corollary} \label{cor:coorbitequivalence}
    Let $H_1, H_2 \leq \mathrm{GL}(d, \R)$ be connectivity-respecting
    with essential frequency supports $\CalO_1 = H_1^T C_1$ and  $\CalO_2 = H^T_2 C_2$ for compact,
    connected sets $C_1 \subseteq \CalO_1$ and $C_2 \subseteq \CalO_2$.
    Let $(\CalO_1)_0$ and $(\CalO_2)_0$ be the connected components containing $C_1$ and $C_2$, respectively.

    With notation as in \Cref{notation}, if $H_1$ and $H_2$ are coorbit equivalent,
    then $\CalO := \CalO_1 = \CalO_2$ and, for each $\xi \in (\CalO_1)_0$ and $\eta \in (\CalO_2)_0$, the transition map 
    \[
      (p_{\eta}^{H_2})_\ast \circ p_\xi^{H_1} : 
      (H_1, d_{W_1}) \to (H_2, d_{W_2})
    \]
    is a quasi-isometry, where $(p_{\eta}^{H_2})_\ast$ is a quasi-inverse for 
    $p_{\eta}^{H_2} : (H_2, d_{W_2}) \to (\CalO, d_{\CalP})$.
\end{corollary}

The necessary condition for coorbit equivalent dilation groups provided by \Cref{cor:coorbitequivalence}
resembles the characterization of coorbit equivalence for irreducibly admissible dilation groups proven in \cite{CoarseGeometryPaper}.
More precisely, \cite[Theorem 4.17]{CoarseGeometryPaper} shows that coorbit equivalence of irreducibly admissible dilation groups
can be characterized through the quasi-isometry property of a \emph{single} transition map. 
For reducible dilation groups, the quasi-isometry property of a single transition map does,
however, \emph{not} characterize coorbit equivalence;
see \Cref{ex:counterexample} below.

\subsection{Equivalence of subgroups}

In this subsection, we apply the results from the previous subsection to study the
coorbit equivalence of matrix groups $H_1 \subseteq H_2$.
We start with the following statement, which is contained in \cite[Lemma 6.4]{FuehrVelthoven}.

\begin{lemma}\label{lem:cocompact}
  Let $H_1 \leq H_2$ be two closed matrix groups such that $H_2/H_1$ is compact.
  Assume that $H_2$ is integrably admissible with essential frequency support $\CalO$.
  Then $H_1$ is integrably admissible with frequency support $\CalO$, and coorbit equivalent to $H_2$. 
\end{lemma}

A converse can now be provided via \Cref{cor:coorbitequivalence}, resulting in the following characterization:

\begin{corollary} \label{cor:cocompact}
  Let $H_1 \leq H_2 \leq \GL(d, \R)$ denote two connectivity-respecting integrably admissible dilation groups
  and such that $H_1 \subseteq H_2$.
  Then $H_1$ and $H_2$ are coorbit equivalent if and only if $H_2/H_1$ is compact. 
\end{corollary}

\begin{proof}
The "if"-direction  holds by \Cref{lem:cocompact}.
For the converse, we assume that $H_1 \subseteq H_2$ are coorbit equivalent.
Then, we get that $\CalO_1 = \CalO_2 =: \CalO$ by \Cref{cor:coorbitequivalence},
where $\CalO_1 = H^T_1 C_1$ and $\CalO_2 = H^T_2 C_2$ denote the essential frequency supports of $H_1$ and $H_2$, respectively.
Since $H_1 \subseteq H_2$, we have $\CalO = H_1^T C_1 \subseteq H_2^T C_1 \subseteq H_2^T \CalO = \CalO$,
so we may fix a single compact, connected set $C \subseteq \CalO$ such that $\CalO = H_1^T C = H_2^T C$.
By \Cref{lem:stabilizer_compactlygenerated}, the stabilizer subgroups $(H_1)_0$ and $(H_2)_0$
of the connected component $\CalO_0$ of $\CalO$ containing $C$ is compactly generated.
For $i=1,2$, fix word metrics $d_{W_i}$ on $H_i$ with $W_i \subseteq H_i$ fulfilling the conditions of \Cref{notation}.

Fix any $\xi \in C$ and let $p_\xi^{H_i} : H_i \to \CalO$ denote the orbit map,
for $i=1,2$.
By \Cref{cor:orbitmap_quasi}, $p_\xi^{H_1} : (H_1, d_{W_1}) \to (\CalO, d_\CalQ)$ is a quasi-isometry,
and similarly for  $p_\xi^{H_2}$.
Let $\left( p_\xi^{H_2}\right)_*: \CalO \to H_2$ denote a quasi-inverse
of $p_\xi^{H_2}$.
By  \Cref{cor:coorbitequivalence}, the map
$
\left( p_\xi^{H_2}\right)_*\circ p_\xi^{H_1} :
  (H_1,d_{W_1}) \to (H_2,d_{W_2}) 
$
is a quasi-isometry.
Condition \ref{enu:QuasiIsometryCondition2} of the quasi-isometry property
provides a finite constant $R_1 > 0$ such that 
\begin{equation}\label{eqn:qd_1}
  \sup_{h_2 \in H_2} \,\,
    \inf_{h_1 \in H_1} \,\,
      d_{W_2} \left(\left( p_\xi^{H_2}\right)_* \circ p_\xi^{H_1}(h_1),h_2 \right)
  \le R_1
  .
\end{equation}
Combining the definition of the orbit maps with the inclusion $H_1 \subseteq H_2$, one obtains that
\[
    \left( p_\xi^{H_2}\right)_* \circ p_\xi^{H_1} (h_1)
    =  \left( p_\xi^{H_2}\right)_* (h_1^{-T} \xi)
    = \left( p_\xi^{H_2}\right)_* \circ p_\xi^{H_2} (h_1), \quad
    h_1 \in H_1.
\]
This observation yields via \Cref{lem:inv_coarse_isom}(ii) that
\begin{equation} \label{eqn:qd_2}
  \sup_{h_1 \in H_1}
    d_{W_2}\left(h_1,\left( p_\xi^{H_2}\right)_* \circ p_\xi^{H_1} (h_1) \right)
  \le R_2
\end{equation}
for a suitable finite constant $R_2 > 0$.

Note that \Cref{eqn:qd_2} implies for arbitrary $h_1 \in H_1$ and $h_2 \in H_2$ that
\begin{align*}
  d_{W_2} (h_1, h_2)
  & \leq d_{W_2} \bigl(h_1, (p_{\xi}^{H_2})_* \circ p_\xi^{H_1} (h_1)\bigr)
         + d_{W_2} \bigl((p_{\xi}^{H_2})_* \circ p_\xi^{H_1} (h_1), h_2\bigr) \\
  & \leq K_2 + d_{W_2} \bigl((p_{\xi}^{H_2})_* \circ p_\xi^{H_1} (h_1), h_2\bigr) 
  .
\end{align*}
Taking the infimum over $h_1 \in H_1$ on both sides, \Cref{eqn:qd_1} then implies
\[
  \inf_{h_1 \in H_1} 
    d_{W_2} (h_1, h_2)
  \leq K_2
       + \inf_{h_1 \in H_1}
           d_{W_2} \bigl((p_{\xi}^{H_2})^{\ast} \circ p_\xi^{H_1} (h_1), h_2\bigr) 
  \leq R_1 + R_2
\]
for arbitrary $h_2 \in H_2$.
In other words, we have shown
\[
  \sup_{h_2 \in H_2} \,\,
    \inf_{h_1 \in H_1} \,\,
      d_{W_2}(h_1,h_2)
  \leq R_1 + R_2. 
\]
Hence, if $B \subseteq H_2$ denotes the $d_{W_2}$-ball with radius $R_1 + R_2$,
then for arbitrary $h_2 \in H_2$ there exists $h_1 \in H_1$ such that $b := h_1^{-1} h_2^{-1} \in B$
and thus $h_2 = b^{-1} h_1^{-1}$, whence $h_2 H_1 = b^{-1} H_1$.
This shows that
\[
  H_2 / H_1
  = \{ b^{-1} H_1 \,\,\colon\,\, b \in \overline{B} \}
  ,
\]
which is compact, since $\overline{B} \subseteq H_2$ is compact.
\end{proof}

\subsection{Equivalence of conjugate subgroups}
As an application of \Cref{thm:char_coorbit_equiv_new}, we show that the property of
coorbit equivalence is preserved under conjugation.
We remark that this could in principle also be proven directly from the definition,
but the proof gets quite tedious.

\begin{corollary}\label{cor:CoorbitEquivalenceConjugation}
    For a closed subgroup $H \leq \GL(d,\R)$ and $A \in \GL(d,\R)$, the following assertions hold:
    \begin{enumerate}
        \item[(i)] If $H$ is integrably admissible, then so is $A^{-1} H A$.
              If $\CalO$ is the essential frequency support of $H$,
              then $\CalO' := A^T \CalO$ is the essential frequency support of $A^{-1} H A$.

        \item[(ii)] If $H$ is connectivity-respecting, then so is $A^{-1} H A$.

        \item[(iii)] If $H_1,H_2 \leq \GL(d,\R)$ are such that $H_1,H_2$ are integrably admissible,
              connectivity-respecting, and coorbit equivalent, then the same holds
              for $A^{-1} H_1 A$ and $A^{-1} H_2 A$.
    \end{enumerate}
\end{corollary}

\begin{proof}
    Throughout the proof, we simply write $H' = A^{-1} H A$.
\\~\\ 
    (i)
    Let $H$ be integrably admissible with essential frequency support $\CalO \subseteq \R^d$,
    say $\CalO = H^T C$ with $C \subseteq \CalO$ compact.
    Then, note for $C' := A^T C$ that $\CalO' := A^T \CalO \subseteq \R^d$ is open and of full measure,
    and
    \[
      (H')^T C'
      = A^T H^T A^{-T} A^{T} C
      = A^T H^T C
      = A^T \CalO
      = \CalO'
      .
    \]
    Lastly, if $K' \subseteq \CalO'$ is a compact set, then also $K := A^{-T} K' \subseteq \CalO$ is compact,
    and, moreover, 
    \begin{align*}
      [K']
      & = \big\{ (h', \xi ') \in H' \times \CalO' \,\,:\,\, ( (h')^{-T} \xi ', \xi ') \in K' \times K' \big\} \\
      & = \big\{ (A^{-1} h A, A^T \xi ) \,\,:\,\, (h,\xi) \in H \times \CalO \text{ and } (h^T \xi, \xi) \in K \times K \big\} \\
      & 
      \subseteq H' \times \CalO'
    \end{align*}
    is compact.
    In combination, this shows that $H'$ is integrably admissible.
    \\~\\
    (ii)
    Let $H$ be connectivity-respecting.
    Then there exists a compact, connected set $C \subseteq \CalO$ with $\CalO = H^T C$
    and such that $H_0 := \{h \in H \,\,:\,\, h^T \CalO_0 = \CalO_0 \}$ is compactly generated,
    where $\CalO_0 \subseteq \CalO$ is the connected component of $\CalO$ containing $C$.
    Let $C' := A^T C$ and $\CalO' := A^T \CalO$.
    Then $C' \subseteq \CalO'$ is compact and connected,
    and, similarly as in the proof of part (i), it follows that $(H')^T C' = \CalO'$.
    Moreover, $\CalO_0 ' := A^T \CalO_0$ is the connected component of $\CalO'$ containing $C'$,
    and we see that
$  
        H_0' 
        = A^{-1} H_0 A
$    
    is compactly generated, since $H_0$ is.
    Hence, $H'$ is connectivity-respecting.
\\~\\
(iii)
    With notation as in \Cref{thm:char_coorbit_equiv_new}, we see that $\CalO_1 = \CalO_2 =: \CalO$,
    and with
    \begin{align*}
        p_1 : H_1 \times C_1 \to \CalO, (h, \xi) \mapsto h^{-T} \xi
        \qquad \text{and} \qquad
        p_2 : H_2 \times C_2 \to \CalO, (h, \xi) \mapsto h^{-T} \xi
    \end{align*}
    that $p_2^\ast \circ p_1 : (H_1 \times C_1, d_{H_1 \times C_1}) \to (H_2 \times C_2, d_{H_2 \times C_2})$
    is a quasi-isometry, where $p_2^\ast$ is a right-inverse for $p_2$.
    Set $H_i ' := A^{-1} H_i A$ and $C_i ' := A^T C_i$, as well as $\CalO_i ' := A^T \CalO_i$
    and $W_i ' := A^{-1} W_i A$ for $i =1,2$.
    Then $\CalO_1 ' \!=\! \CalO_2 ' \!=\! A^T \CalO$, and 
    $d_{W_i '} (g', h') = d_{W_i} (A g' A^{-1}, A h' A^{-1})$ for $g', h' \in H_i '$, so that the maps
    \[
      \tau_i :  H_i \times C_i \to H_i ' \times C_i ', \quad (h,\xi) \mapsto (A^{-1} h A, A^T \xi)
    \]
    are bijective (quasi)-isometries for $i =1,2$.
    
    Finally, for the orbit maps $ p_1 ' : H_1 ' \times C_1 ' \to \CalO '$ and $p_2 ' : H_2 ' \times C_2 ' \to \CalO '$,
    it is easy to see that the map
    $
      (p_2 ')^\ast : 
      \CalO ' \to H_2 ' \times C_2 ', \;
      \eta \mapsto \tau_i (p_2^\ast (A^{-T} \eta))
    $
    is a right-inverse to $p_2 '$. Moreover, 
    \begin{align*}
      ((p_2 ')^\ast \circ p_1 ') (h', \xi ')
      & = (\tau_2 \circ p_2^\ast) (A^{-T} (h')^{-T} \xi ') \\
      & = (\tau_2 \circ p_2^\ast \circ p_1) (A h' A^{-1}, A^{-T} \xi ') \\
      & = [\tau_2 \circ (p_2^\ast \circ p_1) \circ \tau_1^{-1} ] (h', \xi')
      ,
    \end{align*}
    so that $(p_2 ')^\ast \circ p_1 ' = \tau_2 \circ (p_2^\ast \circ p_1) \circ \tau_1^{-1}$ is a quasi-isometry
    as a composition of quasi-isometries.
    Hence, \Cref{thm:char_coorbit_equiv_new} shows that $H_1 '$ and $H_2 '$ are coorbit equivalent.
\end{proof}

\section{Anisotropic Besov spaces and one-parameter groups}
\label{sec:application}

In this section, the results from the previous sections will be used to investigate
the relation between coorbit spaces associated to one-parameter dilation groups and Besov spaces
associated to (possibly anisotropic) matrix dilations introduced in \cite{bownik2005atomic}.

Let $A \in \GL(d, \R)$ be an expansive matrix, i.e., all eigenvalues are strictly greater than one in modulus.
Following \cite{bownik2005atomic}, for $p, q \in [1,\infty]$ and $\alpha \in \R$,
the associated anisotropic Besov space $\dot{B}_{p,q}^{\alpha}(A)$ is defined as the space
\begin{align} \label{eq:besov}
    \dot{B}_{p,q}^{\alpha} (A)
    = \bigg\{
        f \in \mathcal{S}' (\R^d) / \mathcal{P}
        :
        \big\| \big( |\det A|^{\alpha j} \| f \ast \varphi_j \|_{L^p} \big)_{j \in \Z} \big\|_{\ell^q}
        < \infty
      \bigg\},
\end{align}
where $\varphi_j := |\det A|^j \varphi (A^j \cdot)$ for some suitable $\varphi \in \mathcal{S}(\R^d)$
and where $\CalP$ denotes the space of all ($d$-variate) polynomials;
see \cite{bownik2005atomic} for further details.
The coorbit space $\Co(L^{p,q}(\R^d \rtimes \langle A \rangle))$ associated with the cyclic group $\langle A \rangle := \{A^j : j \in \Z \}$ can be identified with an anisotropic Besov space $\dot{B}_{p,q}^{\alpha}(A)$ for some $\alpha = \alpha(p,q) \in \R$, see, e.g., \cite[Example 6.2]{FuehrVelthoven}.

We start out by characterizing the matrix groups whose coorbit spaces coincide
with the isotropic Besov spaces \cite{frazier1985decomposition}, that is,
the Besov spaces $\dot{B}_{p,q}^{\alpha} (A)$ associated to $A = 2 \cdot I_d$.
It was already observed in \cite{grochenig1991describing, groechenig1988unconditional}
that these spaces can be understood as coorbit spaces associated to the (irreducibly admissible)
group $\mathbb{R}^+ \cdot SO(d)$, or to the (integrably admissible) one-parameter group $\mathbb{R}^+ \cdot I_d$.
The following theorem extends this observation by characterizing all potential candidates.

\begin{theorem}\label{thm:isotropic_Besov}
 Let $H \leq \mathrm{GL}(d, \R)$ be integrably admissible.
 \begin{enumerate}
 \item[(i)] Suppose that $H = \exp(\mathbb{R} X)$ for some $X \in \mathbb{R}^{d \times d}$.
            Then $H = \exp(\mathbb{R}X)$ is coorbit equivalent to $\mathbb{R}^+ \cdot I_d$
            if and only if $X = s \cdot I_d + Y$, with $s \not= 0$ and such that
            $\exp(\mathbb{R}Y)$ is relatively compact in $\GL(d, \R)$.
            
 \item[(ii)] Suppose that $H$ is connected.
             Then $H$ is coorbit equivalent to $\mathbb{R}^+ \cdot I_d$ if and only if $H$ is conjugate
             to a noncompact closed subgroup of $\mathbb{R}^+ \cdot SO(d)$.
 \end{enumerate}
 \end{theorem}

 \begin{proof}
    The claim is trivial for $d = 1$, so that we can (and will)
    assume $d > 1$ for what follows.

     For the proof of part (i), assume that $H = \exp(\R X)$. 
     By \cite[Proposition 6.3]{FuehrVelthoven}, the group $H = \exp(\mathbb{R} X)$ is integrably admissible
     if and only if all real parts of eigenvalues of $X$ are either strictly positive or strictly negative.
     In either case, we find that 
     $s := \frac{{\rm trace}(X)}{d} \not= 0$, and after replacing $X$ by a nonzero scalar multiple, we may assume $s=1$.
     By \Cref{ex:connectivity_respecting}, the group $H$ is connectivity-respecting,
     since $H = \exp(\R X)$ is connected and since
     \cite[Proposition 6.3]{FuehrVelthoven} shows that the essential
     frequency support is $\CalO = \R^d \setminus \{0\}$, which is connected, since $d > 1$.

     Define $Y := X- s \cdot I_d$, so that in particular ${\rm trace}(Y) = 0$.
     In addition, let $A := \exp(X)$ and $B := e \cdot I_d$.
     Then, by \Cref{lem:cocompact}, the one-parameter group $\exp(\mathbb{R} X)$
     is coorbit equivalent to the cyclic group $\langle A \rangle$ generated by $A$,
     and $\mathbb{R}^+ \cdot I_d$ is coorbit equivalent to the cyclic group $\langle B \rangle$ generated by $B$. Note that \Cref{lem:cocompact} also implies that $\langle A \rangle $ and $\langle B \rangle $ are integrably admissible.
     The coorbit spaces associated to these cyclic groups
     can be canonically identified with homogeneous Besov spaces $\dot{B}_{p,q}^\alpha(A)$
     and $\dot{B}_{p,q}^\alpha(B)$ with $\alpha = \frac{1}{2} - \frac{1}{q}$; see, e.g., \cite[Example 6.2]{FuehrVelthoven}.
     By our choice of normalizations and since $\det (\exp(M)) = e^{\mathrm{trace} (M)}$, we see
     \[
        \epsilon
        := \epsilon(A,B)
        := \frac{\ln(|\det(A)|)}{\ln(|\det(B)|)}
        = s
         = 1,
    \]
     and thus it follows from \cite[Corollary 6.5, Definition 4.5, Remark 4.9, and Lemma 4.8]{FuehrCheshmavar}
     that the equality $\dot{B}_{p,q}^\alpha(A) = \dot{B}_{p,q}^\alpha(B)$ holds if and only if
     \[
       \sup_{k \in \Z}
         \| A^{-k} B^{\lfloor \epsilon k \rfloor} \|
       < \infty .
     \]
     Since $A^{-k} B^{\lfloor \epsilon k \rfloor} = \exp(- k Y)$ for $k \in \Z$, the latter condition is equivalent
     to the cyclic group $\langle \exp(Y) \rangle$ having compact closure in ${\rm GL}(d,\mathbb{R})$.
     Since $\langle \exp(Y) \rangle = \exp (\Z Y)$ is cocompact in $\exp{\mathbb{R}Y}$,
     relative compactness of $\langle \exp(Y) \rangle$ is equivalent to relative compactness of $\exp{\mathbb{R}Y}$ in $\GL(d, \R)$,
     which completes the proof of assertion (i). 

     For the proof of assertion (ii), assume first that $H$ is connected and coorbit equivalent to the one-parameter group $\mathbb{R}^+\cdot I_d$.
     Then, by \Cref{cor:coorbitequivalence}, the group $H$ is quasi-isometric to $\mathbb{R}$.
     \Cref{prop:char_lin_grwth} entails therefore the existence of a cocompact closed noncompact one-parameter group $\exp(\mathbb{R}X) \subseteq H$.
     \Cref{lem:cocompact} yields that this subgroup is coorbit equivalent to $H$, and hence to $\mathbb{R}^+\cdot I_d$.
     Assertion (i) entails that (possibly after renormalization) $X = I_d + Y$ and that $\exp(\mathbb{R}Y)$ is relatively compact in $\GL(d,\R)$.
     This implies in particular that the set $\{ \det (\exp(t Y)) : t \in \R \} = \{ \exp(\mathrm{trace}(t Y)) : t \in \R \}$
     is relatively compact in $\R \setminus \{ 0 \}$, and thus $\operatorname{trace} Y = 0$. 

     We next show that $H$ can be written as the semidirect product $H = K \rtimes \exp(X)$, where $K$ is the closed normal subgroup $K := \{ h \in H: {\rm det}(h) = 1 \}$. Moreover, we show that $K$ is compact and connected. For this, first note that since $H$ is connected, ${\rm det}(h)>0$ for all $h \in H$. 
     By normalization of $X$, for any $h \in H$ we have 
     \begin{equation} \label{eqn:factor_simil}
       \det (\exp(sX))
       = \exp(\operatorname{trace}(s X))
       = \det (h),\,\,  \mbox{ with } s = \frac{\ln(\det(h))}{d}\,.
     \end{equation}
     This shows that $k := \exp(-sX) h \in K$, and we obtain the unique factorization
     \begin{equation}
       h = \exp(s X) k, \; k \in K ,
       \label{eq:Factorization}
     \end{equation}
and hence $H = K \rtimes \exp(\mathbb{R} X)$. 

For the compactness of $K$,
 let $(k_n)_{n \in \N} $ be a sequence in $K$.
         Since $H / \exp(\R X)$ is compact, it follows from \cite[Lemma (2.46)]{FollandAHA}
         that there exists a compact set $C \subseteq H$ such that $H = C \exp(\R X)$.
         By compactness, we have $0 < \alpha \leq \det(c) \leq \beta < \infty$ for all $c \in C$ and suitable $\alpha,\beta$.
         We can hence write $k_n = c_n \exp(t_n X)$ for suitable $c_n \in C$ and $t_n \in \R$.
         Note that 
         \[1 = \det (k_n) = \det(c_n) \det(\exp(t_n X)) = \det(c_n) \exp(\operatorname{trace} (t_n X)), \]
         and thus $\frac{1}{\beta} \leq \exp(t_n d)  \leq \frac{1}{\alpha}$.
         Therefore, $(t_n)_{n \in \N}$ is bounded, so that for a subsequence  $t_{n_\ell} \to t \in \R$ and
         $c_{n_\ell} \to c \in C$.
         In turn, this implies $k_{n_\ell} = c_{n_\ell} \exp(t_{n_\ell} X) \to c \exp(t X) =: k$,
         with $k \in K$, since $K \subseteq H$ is closed. Thus, $K$ is a compact subgroup.

 To show that $K$ is connected, we note that the map $H \to K$, $h \mapsto k$ explicitly determined in (\ref{eqn:factor_simil}) and (\ref{eq:Factorization}) is clearly continuous, and equals the identity on $K$. In particular this map is onto, and hence $K$ is the continuous image of the connected group $H$, and therefore itself connected. 
         
    Since $H = K \rtimes \exp(\mathbb{R} X)$, it follows, in particular,  that $\exp(\R X)$ normalizes $K$, that is, $\exp(t X) K \exp(-t X) \subseteq K$ for all $t \in \R$.
     Now, since $Y = X - I_d$ and since $X$ and $I_d$ commute,
     we have $\exp(t Y) = \exp(t \cdot (X - I_d)) = \exp(t X) \exp(-t I_d) = e^{-t} \exp(t X) = \exp(t X) e^{-t}$
     for all $t \in \R$.
     In view of this identity, it follows that also $\exp(\mathbb{R}Y)$ normalizes $K$,
     and the same then holds for the closure $\overline{\exp (\mathbb{R}Y)}$ of $\exp(\R Y)$ in $\GL(d, \R)$.
     In combination with the fact that $K$ is compact and connected, it follows that $\overline{\exp(\mathbb{R}Y)} \cdot K \le {\rm GL}(d,\mathbb{R})$ is a compact connected subgroup.
     Since any compact subgroup of $\mathrm{GL}(d, \R)$ is conjugate to a subgroup of $\mathrm{O}(d)$ (see, e.g., \cite[Proposition 6.3.6]{abbaspour2007basic}) it follows that $\overline{\exp(\mathbb{R}Y)} \cdot K$ is conjugate
     to a subgroup of $\mathrm{O}(d)$, and (by connectedness) even to a subgroup of $\mathrm{SO}(d)$.
     In turn, using again that $\exp(t X) = e^t \exp(t Y)$,
     this implies that $H = K \rtimes \exp(\R X)$ is conjugate to a subgroup of $\mathbb{R}^+ \cdot SO(d)$, and it is clearly noncompact. 
     
     For the converse implication, note first that it suffices to show that a connected closed noncompact subgroup $H$ of $\mathbb{R}^+ \cdot \mathrm{SO}(d)$ is coorbit equivalent to $\mathbb{R}^+ \cdot I_d$, because then, by \Cref{cor:CoorbitEquivalenceConjugation}, also any conjugate
     $A^{-1} H A$ is coorbit equivalent to $A^{-1} (\mathbb{R}^+ \cdot I_d)  A = \mathbb{R}^+ \cdot I_d$ for $A \in \mathrm{GL}(d, \mathbb{R})$. Hence, let $H$ be a connected closed noncompact subgroup $H$ of $\mathbb{R}^+ \cdot \mathrm{SO}(d)$. Note that the group $\mathbb{R}^+ \cdot \mathrm{SO}(d)$ is isomorphic to  $\R^+ \times \mathrm{SO}(d)$ via  the map $(r,h) \mapsto r h$, and hence it follows by elementary Lie theory that the Lie group of $\mathbb{R}^+ \cdot \mathrm{SO}(d)$
     is the sum of $\mathrm{span} \{ I_d \}$ and the Lie algebra $\mathfrak{so}(d)$ of $\mathrm{SO}(d)$.
     Since $H$ is not contained in $\mathrm{SO}(d)$ and since it is generated by the image of the exponential map (because $H$ is connected),
     this image must contain a one-parameter group $\exp(\mathbb{R} X) \subseteq H$ which is not contained in $\mathrm{SO}(d)$.
     We may then normalize the infinitesimal generator to obtain $X = I_d + Y$, where $Y \in \mathfrak{so}(d)$. 
     As in the proof of the converse direction, for every $h \in H$, we have 
     \[
       h = \exp(s X) k, \; k \in K .
     \]
     where $s$ is given by (\ref{eqn:factor_simil}), and 
     \[
       K = \{ h \in H : {\rm det}(h) = 1 \}~. 
     \]
     The assumptions on $H$ then entail that $K$ is a closed subgroup of $\mathrm{SO}(d)$, in particular compact. 
     This establishes that $\exp(\mathbb{R} X) \subseteq H$ is cocompact,
     so that \Cref{lem:cocompact} shows that $H$ and $\exp(\R X)$ are coorbit equivalent.
     Coorbit equivalence to $\mathbb{R}^+ \cdot I_d$ then follows by part (i). 
\end{proof}

The next theorem clarifies the coorbit equivalence of two
somewhat extreme subclasses of dilation groups:
Irreducibly admissible dilation groups on the one hand, and one-parameter subgroups on the other.
In particular, it implies that the use of reducibly acting dilation groups is essential
for treating the full scale of anisotropic Besov spaces within the setting of coorbit spaces.
The uniqueness of the essential frequency support, noted in Theorem \ref{thm:DualOrbitUniqueness},
is an essential ingredient of the following proof. 

\begin{theorem} \label{thm:oneparameter}
  Let $H_1 \leq {\rm GL}(d, \mathbb{R})$ be a one-parameter integrably admissible dilation group,
  and let $H_2 \leq  {\rm GL}(d, \mathbb{R})$ be an irreducibly admissible dilation group (see \Cref{ex:dilations}).
  If $H_1$ and $H_2$ are coorbit equivalent, then $H_1$ is coorbit equivalent to  $\mathbb{R}^+ \cdot I_d$.
\end{theorem}

\begin{proof}
Let $Y \in \R^{d \times d}$ denote the infinitesimal generator of $H_1$, i.e., $H_1 = \exp(\R Y)$.
By \cite[Proposition 6.3]{FuehrVelthoven}, the group $H_1 = \exp(\mathbb{R} Y)$
is integrably admissible if and only if all real parts of eigenvalues of $Y$ are either strictly positive
or strictly negative.
We may therefore assume that they all have positive real part,
which yields that the matrix $A :=\exp(Y)$ is expansive, in the sense that all its eigenvalues are strictly greater than one in modulus.

By \Cref{lem:cocompact}, $H_1 = \exp(\mathbb{R}Y)$ is coorbit equivalent
to the cocompact cyclic subgroup $\langle A \rangle = \exp(\mathbb{Z}Y)$,
and the latter is integrably admissible.
For proving the claim, it suffices therefore to show that $\langle A \rangle $ is coorbit equivalent to $\mathbb{R}^+ \cdot I_d$.
For this, recall that the coorbit spaces $\Co(L^{1,1}(\R^d \rtimes \langle A \rangle ))$
associated to $\langle A \rangle$ can be identified with anisotropic Besov spaces
$\dot{B}_{1,1}^\alpha(A)$ for a suitable $\alpha \in \R$; see, e.g., \cite[Example 6.2]{FuehrVelthoven}.
Define the matrix group
\[
  \mathcal{S}_{H_1}
  := \{
       C \in {\rm GL}(d, \mathbb{R})
       :
       \dot{B}_{1,1}^\alpha (C^{-1} A C) =  \dot{B}_{1,1}^\alpha (A)
     \}. 
\] 

Using \cite[Lemmata 7.7 and  7.8]{FuehrCheshmavar} (see also \cite[Remark 7.11]{FuehrCheshmavar}),
there exists an expansive matrix $B$ such that $\dot{B}_{p,q}^{\beta}(A) = \dot{B}_{p,q}^{\beta} (B)$
for all $p,q \in [1,\infty]$ and $\beta \in \R$, and, in addition,
$B$ is in \textit{expansive normal form}, that is, $B$ has only positive eigenvalues,
with ${\rm det}(B) = 2$.
In view of the identification of the coorbit spaces $\Co(L^{p,q}(\R^d \rtimes \langle A \rangle ))$
and $\Co(L^{p,q}(\R^d \rtimes \langle B \rangle ))$ with the homogeneous Besov spaces
$\dot{B}_{p,q}^\beta (A)$ and $\dot{B}_{p,q}^\beta (B)$ with a certain $\beta = \beta(p,q) \in \R$
(see \cite[Example 6.2]{FuehrVelthoven}), this implies that $\langle A \rangle$ and $\langle B \rangle$
are coorbit equivalent. For further use below, we note at this point that the identification of anisotropic Besov spaces with suitable coorbit spaces also allows to derive the alternative description 
\begin{equation} \label{eqn:alt_def_S}
 \mathcal{S}_{H_1}
  = \{
       C \in {\rm GL}(d, \mathbb{R})
       :
       C^{-1} H_1 C \mbox{ is coorbit equivalent to }  H_1     \}. 
\end{equation}

By  \cite[Corollary 6.5]{FuehrCheshmavar}, the fact that $\dot{B}_{1,1}^\alpha (A) = \dot{B}_{1,1}^\alpha(B)$ implies that $A$ and $B$ are equivalent in the sense of \cite[Corollary 6.5]{FuehrCheshmavar}. Since also $C^{-1} A C$ and $C^{-1} B C$ are equivalent in the same sense for every $C \in \GL(d, \R)$, this implies by another application of \cite[Corollary 6.5]{FuehrCheshmavar} that also $\dot{B}_{1,1}^\alpha (C^{-1} AC) = \dot{B}_{1,1}^\alpha(C^{-1} BC)$. Hence,
\[
  \mathcal{S}_{H_1}
  = \{
      C \in {\rm GL}(d, \mathbb{R})
      :
      \dot{B}_{1,1}^\alpha (C^{-1} B C) =  \dot{B}_{1,1}^\alpha (B)
    \}.
\] 
Since both $B$ and $C^{-1}BC$ are in expansive normal form,
for every invertible matrix $C$, an application of \cite[Theorem 7.9]{FuehrCheshmavar} implies that
\[
  \mathcal{S}_{H_1}
  = \{ C \in {\rm GL}(d,\mathbb{R}): C^{-1} B C = B \}
  = Z(B), 
\]
where $Z(B) \subseteq {\rm GL}(d,\mathbb{R})$ denotes the \textit{centralizer} of $B$. 

Suppose now that $H_2$ is an irreducibly admissible matrix group that is coorbit equivalent to $H_1$. 
By \Cref{thm:coorbit_equiv_dual_orbits}, 
$H_2$ has the same frequency support as $H_1$,
which is
\[
  \mathcal{O} = \mathbb{R}^d \setminus \{ 0 \}
\]
cf. \cite[Example 6.2]{FuehrVelthoven}. 

 We now prove that $H_2 \subseteq \mathcal{S}_{H_1}$. For this, let $h \in H_2$. Then $h^{-1} H_2 h$ is trivially coorbit equivalent to $H_2$, and therefore also to $H_1$. On the other hand, $h^{-1} H_2 h$ is coorbit equivalent to $h^{-1} H_1 h$ by \Cref{cor:CoorbitEquivalenceConjugation} (iii); observe that both groups are connectivity respecting by Examples \ref{ex:connectivity_respecting} (1) and (2).
In combination, this yields that $H_1$ and $h^{-1} H_1 h$ are coorbit equivalent, for all $h \in H_2$, and thus Equation (\ref{eqn:alt_def_S}) yields
\[
  H_2 \subseteq \mathcal{S}_{H_1} = Z(B)~. 
\]

Arguing by contradiction, assume that the group $\langle A \rangle$ is not coorbit equivalent to $\mathbb{R}^+ \cdot I_d$.
Since $A$ and $B$ are coorbit equivalent, this is precisely the case if $ \langle B \rangle$ is not coorbit equivalent to $\mathbb{R}^+ \cdot I_d$.  
The expansive normal form matrix associated to the latter group is given by $2^{1/d} \cdot I_d$.
Hence, $B \neq 2^{1/d} I_d$ by \cite[Theorem 7.9]{FuehrCheshmavar}.
Now the fact that $B \not= 2^{1/d} \cdot I_d$ forces, for any eigenvalue $\lambda$ of $B^T$
with associated eigenspace $E_\lambda$ that $E_\lambda \subsetneq \mathbb{R}^d$; otherwise, $B^T$ and hence $B$
would be a multiple of the identity and thus $B = 2^{1/d} I_d$ since $\det B = 2$.
The definition of $Z(A)$ then immediately entails $C^T E_\lambda \subseteq E_\lambda$,
for all $C \in Z(A)$. 
As a consequence, one gets 
\[
  Z(B)^T E_\lambda \subseteq E_\lambda .
\]
In particular, the dual action of $Z(B)$ on $\CalO = \mathbb{R}^d \setminus \{0 \}$ cannot be transitive, which implies that the dual action of the subgroup $H_2 \subseteq Z(B)$ is also not transitive.
This yields the desired contradiction, and completes the proof.
\end{proof}

Lastly, we provide the example mentioned in the discussion below \Cref{cor:coorbitequivalence},
which shows the difference of coorbit equivalence between irreducible and reducible dilation groups. 

\begin{example} \label{ex:counterexample}

Define
\[
  A = \begin{pmatrix}
        3 & 0 & 0 \\
        0 & 2 & 0 \\
        0 & 0 & 2
      \end{pmatrix}
  \qquad \text{and} \qquad
  B = \begin{pmatrix}
        2 & 0 & 0 \\
        0 & 2 & 0 \\
        0 & 0 & 3
      \end{pmatrix}
\]
as well as
\[
  A_0 = \begin{pmatrix}
          \ln(3) & 0      & 0 \\
          0      & \ln(2) & 0 \\
          0      & 0      & \ln(2)
        \end{pmatrix}
  \qquad \text{and} \qquad
  B_0 = \begin{pmatrix}
          \ln(2) & 0      & 0 \\
          0      & \ln(2) & 0 \\
          0      & 0      & \ln(3)
        \end{pmatrix}
  .
\]
Furthermore, let $H_1 := \exp(\R A_0)$ and $H_2 := \exp (\R B_0)$
and note that if we define
\[
  A(t)
  := \begin{pmatrix}
       3^t & 0   & 0 \\
       0   & 2^t & 0 \\
       0   & 0   & 2^t
     \end{pmatrix}
  \qquad \text{and} \qquad
  B(t)
  := \begin{pmatrix}
       2^t & 0   & 0 \\
       0   & 2^t & 0 \\
       0   & 0   & 3^t
     \end{pmatrix}
\]
then $A : \R \to H_1, t \mapsto A(t)$ and $B : \R \to H_2, t \mapsto B(t)$
are isomorphisms of topological groups.
In particular, 
$H_1 = \{ A(t) \,\,:\,\, t \in \R \}$ and $H_2 = \{ B(t) \,\,:\,\, t \in \R \}$.
Since the real parts of all eigenvalues of $A_0, B_0$ are strictly positive,
\cite[Proposition 6.3]{FuehrVelthoven} shows that the groups $H_1, H_2$
are integrably admissible with associated frequency support
$ \CalO = \R^3 \setminus \{ 0 \}$.
In addition, since $H_1, H_2$ and $\CalO$ are connected, it follows by \Cref{ex:connectivity_respecting}(a) that
$H_1, H_2$ are connectivity-respecting, and the stabilizer subgroups of $H_1$ and $H_2$
of the connected components of $\CalO$ are $H_1$ and $H_2$, respectively.
As such, they are compactly generated, and open, precompact, symmetric generating sets for $H_1$ and $H_2$ are given by
$W_1 := \{ A(t) \,\,:\,\, -1 < t < 1 \}$ and $W_2 := \{ B(t) \,\,:\,\, -1 < t < 1 \}$, respectively.

We first show that $H_1$ and $H_2$ are not coorbit equivalent.
Let $p, q \in [1,\infty]$.
By \cite[Corollary 6.7]{FuehrVelthoven}, the coorbit space $\Co (L^{p,q}(\R^3 \rtimes H_1))$ coincides
with the anisotropic, homogeneous Besov space $\dot{B}^{p,q}_{\alpha(p,q)}(A)$
for a certain $\alpha(p,q) \in \R$.
Similarly, $\Co (L^{p,q}(\R^3 \rtimes H_2))$ coincides with  $\dot{B}^{p,q}_{\beta(p,q)}(B)$
for a certain $\beta(p,q) \in \R$.
Since $A,B$ are expansive matrices with $\det A = \det B$, but $A \neq B$,
it follows by \cite[Theorem 7.9(a)]{FuehrCheshmavar} that $A,B$ are not equivalent
in the terminology of \cite{FuehrCheshmavar}.
By \cite[Lemma 6.2]{FuehrCheshmavar}, this means that
the homogeneous covers $\CalQ, \CalP$ associated to $A$ and $B$ are not weakly equivalent.
In turn, this yields (with a similar proof as in \cite[Theorem 6.4]{FuehrCheshmavar})
that the homogeneous, anisotropic Besov spaces
$\dot{B}_{p,q}^{\alpha}(A)$ and $\dot{B}_{p,q}^{\beta}(B)$
never coincide (except perhaps for $(p,q) = (2,2)$).
All in all, this shows that $H_1$ and $H_2$ are not coorbit equivalent.

We now finally show that for a specific choice of $\xi \in \CalO$, an associated transition map is a quasi-isometry.
Choose $\xi := (0,1,0)^T \in \CalO$.
Then the orbit maps are given by
\[
  p_{\xi}^{H_1} : \quad
  H_1 \to \CalO, \quad
  A(t) \mapsto [A(t)]^{-T} \xi = \begin{pmatrix} 0 \\ 2^{-t} \\ 0 \end{pmatrix}
\]
and
\[
  p_{\xi}^{H_2} : \quad
  H_2 \to \CalO, \quad
  B(t) \mapsto [B(t)]^{-T} \xi = \begin{pmatrix} 0 \\ 2^{-t} \\ 0 \end{pmatrix}
  .
\]
Let $\CalQ$ and $\CalP$ be the covers associated to $H_1$ and $H_2$
according to part \ref{enu:Notationcover} of \Cref{notation}.
By \Cref{cor:orbitmap_quasi}, it follows that
$p_{\xi}^{H_2} : (H_2, d_{W_2}) \to (\CalO, d_{\CalP})$ is a quasi-isometry.
As seen in \Cref{sec:quasi-isometry}, this implies that there exists a quasi-inverse
$p_\ast : \CalO \to H_2$ to $p_{\xi}^{H_2}$.
We now define a modified quasi-inverse through
\[
  (p_{\xi}^{H_2})_* : \quad
  \CalO \to H_2, \quad
  \begin{pmatrix} x \\ y \\ z \end{pmatrix}
  \mapsto \begin{cases}
            B(-\log_2 (y)),      & \text{if } x = z = 0 \text{ and } y > 0, \\
            p_\ast ( (x,y,z)^T), & \text{otherwise} .
          \end{cases}
\]
Then, for $y>0$,
\[
  p_\xi^{H_2} ( (p_\xi^{H_2})_* (0,y,0))
  = p_\xi^{H_2} (B(-\log_2 (y)))
  = \begin{pmatrix} 0 \\ 2^{- (-\log_2(y))} \\ 0 \end{pmatrix}
  = \begin{pmatrix} 0 \\ y \\ 0 \end{pmatrix}
  ,
\]
and this easily implies that $(p_\xi^{H_2})_*$ is indeed a quasi-inverse for $p_\xi^{H_2}$.

Finally, note for $\phi := (p_\xi^{H_2})_* \circ p_\xi^{H_1}$ that
\[
  \phi(A(t))
  = (p_\xi^{H_2})_{\ast} \bigl( (0,2^{-t},0)^T \bigr)
  = B( - \log_2 (2^{-t}))
  = B(t)
  ,
\]
which implies that $\phi : H_1 \to H_2$ is an isomorphism,
and this easily shows that it is a quasi-isometry as a map
$\phi : (H_1, d_{W_1}) \to (H_2, d_{W_2})$.

\end{example}

\appendix

\section{Auxiliary results}
\label{sec:TechnicalResults}
This section provides various auxiliary results that are used in the main text and appendices. As we could not locate a convenient reference in the literature, we provide their short proofs.

\begin{lemma}\label{lem:NeighborhoodConnected}
  Let $X$ be a topological space, let $C \subseteq X$ be connected, and for each
  $c \in C$, let $U_c \subseteq X$ be connected with $c \in U_c$.
  Then $\bigcup_{c \in C} U_c$ is connected.

  In particular, if $X$ is a normed vector space and $C \subseteq X$ is a connected set, then
  the $\eps$-neighborhood
  \[
    B_\eps (C)
    = \big\{ x \in X : \mathrm{dist}(x, C) < \eps \big\}
    = \bigcup_{c \in C}
        B_\eps (c)
  \]
  is connected as well.
\end{lemma}

\begin{proof}
  Let $Y = \bigcup_{c \in C} U_c$, $c_0 \in C$ arbitrary, and let $A \subseteq Y$ denote the connected component of $c_0$ in $Y$. Then connectedness of $C$ implies $C \subseteq A$, and $A$ is also the connected component of any $c \in C$. Connectedness of $U_c$ then entails $U_c \subseteq A$. In summary
  \[ Y \supseteq A \supseteq \bigcup_{c \in C} U_c = Y~, \]
  which shows connectedness of $Y$. 
%
%
\end{proof}

\begin{lemma}\label{lem:ConnectedSuperSet}
  Let $U \subseteq \R^d$ be open and connected.
  Then, for each compact set $C \subseteq U$, there exists a compact, connected set
  $K \subseteq U$ with $C \subseteq K$.
\end{lemma}

\begin{proof}
  For each $x \in C \subseteq U$, since $U$ is open, there exists a radius $r_x > 0$
  satisfying $\overline{B}_{r_x} (x) \subseteq U$.
  Then $\bigl(B_{r_x}(x)\bigr)_{x \in C}$ is an open cover of $C$, so that by compactness there exist
  finitely many $x_1,\dots,x_n \in C$ such that $C \subseteq \bigcup_{i=1}^n B_{r_i}(x_i)$,
  where $r_i := r_{x_i}$.

  Since $U$ is open and connected and hence path connected,
  there exists for each $i \in \{ 1,\dots,n \}$ a continuous map
  $\gamma_i : [0,1] \to U$ satisfying $\gamma_i (0) = x_1$
  and $\gamma_i (1) = x_i$.
  Define
  \[
    K
    := \bigcup_{i=1}^n \overline{B}_{r_i} (x_i)
       \cup \bigcup_{i=1}^n \gamma_i ([0,1])
    .
  \]
  Then $K \subseteq U$ is compact and satisfies $K \supseteq C$.
Moreover, by construction of $K$, the path component of $x_1$ in $K$ contains $x_1,\ldots,x_n$, and then also each $\overline{B}_{r_i}(x_i)$. Hence it contains all of $K$, which means that $K$ is connected.
\end{proof}

In the following lemma and its proof, all claims on quotient groups are with respect to the usual quotient topology. 

\begin{lemma}\label{lem:CompactQuotient}
    Let $G$ be a locally compact Hausdorff group. Suppose that $L,H \leq G$ are closed subgroups satisfying
    $L \subseteq H$ and such that $G / L$ is compact.
    Then also $H / L$ is compact.
\end{lemma}

\begin{proof}
    First, since the projection $\pi : G \to G/L$ is a continuous open map, it follows that the subspace topology on $H / L$ as a subspace of $G / L$
    agrees with the quotient topology on $H / L$, see, e.g., \cite[Chapter VI, Theorem 2.1]{Dugundji}

    For proving the claim, we first show that $\pi(G \setminus H) \cap \pi(H) = \emptyset$.
    Arguing by contradiction, suppose there exist $g \in G \setminus H$ and $h \in H$ with
    $g L = \pi(g) = \pi(h) = h L$.
    Then, since $L \subseteq H$, it follows that $g \in h L \subseteq H$, in contradiction to $g \in G \setminus H$.
 Since $\pi : G \to G / L$ is surjective, the fact that $\pi(G \setminus H) \cap \pi(H) = \emptyset$ implies that the complement of $H / L$ in $G / L$ is given by
    \[
      (H / L)^c
      = \pi(G \setminus H)
      .
    \]
    Since $G \setminus H \subseteq G$ is open and $\pi$ is an open map, this implies that $(H / L)^c$ is open in $G / L$,
    so that $H / L$ is closed in the compact set $G / L$, and hence compact as well.
\end{proof}

\begin{lemma}\label{lem:OneParameterGroupClosedness}
    Let $G$ be a first countable locally compact Hausdorff group 
    and let $\gamma : \R \to G$ be a continuous homomorphism.

    Then either the closure $\overline{\gamma(\R)}$ is compact, or $\gamma(\R) \subseteq G$
    is closed and $\gamma : \R \to \gamma(\R)$ is a homeomorphism.
\end{lemma}

\begin{proof}
Suppose that the closure $\overline{\gamma(\R)}$ is not compact. Then, by \cite[Chapter XVI, Proposition 2.3]{HochschildStructureOfLieGroups}, 
the map $\gamma : \R \to \gamma(\R)$ is a homeomorphism, and thus it remains to show that  $\gamma(\R) \subseteq G$ is closed. For this,  let $(h_n)_{n \in \N} \subseteq \gamma(\R)$ be a sequence with $h_n \to g$ for some $g \in G$.
Since $G$ is first countable, it is enough to show that $g \in \gamma(\R)$.
For proving this, note that since $\gamma$ is a homeomorphism onto its range,
there exists an open set $V \subseteq G$ with $\gamma((-1,1)) = V \cap \gamma(\R)$.
This implies that $\gamma(0) = e_G \in V$, so that $V$ is a unit neighborhood.
Thus, there exists a compact unit neighborhood $U \subseteq G$ with $U = U^{-1}$ and $UU \subseteq V$.
By assumption, $h_n \to g$, and hence there exists $n_0 \in \N$ such that $h_n \in g U$ for all $n \geq n_0$.
Hence,
\[
  h_{n_0}^{-1} h_m \in U^{-1} g^{-1} g U = U^{-1} U \subseteq V \subseteq \gamma([-1,1])
  \qquad \text{for all} \quad \, m \geq n_0
  .
\]
Since $\gamma([-1, 1])$ is closed, letting $m \to \infty$, it follows that $h_{n_0}^{-1} g \in \gamma([-1,1]) \subseteq \gamma(\R)$. As $h_{n_0} \in \gamma(\R)$, this also yields that $g \in \gamma(\R)$, which completes the proof.
\end{proof}

\section{Postponed proofs}
\label{sec:postponedproofs}

This section consists of two proofs of results stated in \Cref{sec:ConnectivityRespectingGroups}.

\begin{proof}[Proof of \Cref{lem:inv_coarse_isom}]
Assume that 
\[
  \sup_{y \in Y}
    d_Y\bigl(f_1(f_2(y)),y\bigr)
  = M
  < \infty,
\]
and let $R_1,R_2,R_3$ denote the constants provided by assumptions $(q_1)$ and $(q_2)$ on $f_1$.
Then, for $y,y' \in Y$,
\begin{align*}
  d_X(f_2(y),f_2(y'))
  & \leq R_1  d_Y \bigl(f_1(f_2(y)),f_1(f_2(y')) \bigr) + R_1 R_2 \\
  & \leq R_1 \left( d_Y(f_1(f_2(y)),y) + d_Y(y,y') + d_Y(y',f_1(f_2(y'))) \right) + R_1 R_2 \\
  & \leq R_1 d_Y(y,y') + 2 R_1 M + R_1 R_2 .
\end{align*}
An analogous computation establishes 
\[
   d_X(f_2(y),f_2(y'))
   \ge R_1^{-1} d_Y(y,y') - 2 R_1^{-1} M - R_1^{-1} R_2,
\]
and thus $(q_1)$ is verified for $f_2$.
For the verification of $(q_2)$ we once again invoke $(q_1)$ for $f_1$ and get for all $x \in X$ that
\begin{equation}\label{eqn:quasi_inv}
  d_X(x,f_2(f_1(x)))
  \le R_1 \cdot d_Y \bigl( f_1(x),f_1(f_2(f_1(x))) \bigr) + R_1 R_2
  \le R_1 M + R_1 R_2.
\end{equation}
This implies that $(q_2)$ holds for $f_2$, and thus finishes the proof of $(i)$. Part (ii) also follows directly from 
\eqref{eqn:quasi_inv}.
\end{proof}

\begin{proof}[Proof of \Cref{lem:DifferentWsAreCoarselyEquivalent}]
  Property~\ref{enu:QuasiIsometryCondition2} is clearly satisfied, since the identity is surjective.
  We only prove one direction of the estimate in condition \ref{enu:QuasiIsometryCondition1};
  the other part follows by symmetry.

  Since $V$ is a symmetric generating set for $H_0$, we have $H_0 = \bigcup_{n=1}^\infty V^n$. Hence, by compactness of $\overline{W} \subseteq H_0$, 
  there exists some $N \in \N$ such that $\overline{W} \subseteq V^N$.
  For proving the claim, let $x,y \in H$ be arbitrary and set $m := d_W (x,y) \in \N_0 \cup \{ \infty \}$.
  Note that trivially $d_V(x,y)  \leq N \cdot d_W (x, y)$ whenever $m = 0$ or $m = \infty$, and thus it remains to consider the case  $m \in \N$. In this case,  $x^{-1} y \in W^m \subseteq (V^N)^m \subseteq V^{m N}$ and hence
  \[
    d_V (x,y) \leq m N = N  d_W (x,y)
    ,
  \]
  which completes the proof.
\end{proof}

\section{Connected Lie groups quasi-isometric to \texorpdfstring{$\mathbb{R}$}{ℝ}}

This section is devoted to characterizing connected Lie groups that are quasi-isometric to the real line $\mathbb{R}$.
We expect these results to be folklore, but were unable to locate a convenient reference for them.
In the interest of a self-contained presentation we provide proofs, relying on various sources, most notably \cite{MR1825980,MR4079363}. To make the proof accessible also to readers with limited background in Lie theory,
we provide more detail than is perhaps necessary for specialists.

We start by introducing the notion of growth on locally compact groups.
Throughout this section, we will concentrate on connected Lie groups.

\begin{definition}
    Let $G$ be a connected Lie group.
    For a relatively compact, symmetric open unit neighborhood $U \subseteq G$,
    let $d_U$ be the associated word metric,
    and write
    \[
      U_r := \{ g \in G : d(g,e_G) \leq r \} = U^{\lfloor r \rfloor}
      \quad \text{for} \quad
      r > 0.
    \]
    The \textit{growth function} associated to $U$ and $G$ is defined by 
    \[
        \nu_{U,G} : \quad
        (0,\infty) \to [0,\infty), \quad
        \nu_{U,G}(r) := \mu_G(U_r).
    \]
\end{definition}

Given two nondecreasing functions $\nu, \tilde{\nu}: (0,\infty) \to [0,\infty)$,
following \cite[Definition 3.D.3]{cornulier2016metric},
we write $\nu \preceq \tilde{\nu}$ if there exist constants
$a,b > 0, c \geq 0$ such that for all $r \in (0,\infty)$
\[
  \nu(r) \leq a \tilde{\nu}( b r + c) + c~.
\] 
We write $\nu \simeq \tilde{\nu}$ if both $\nu \preceq \tilde{\nu}$ and $\tilde{\nu} \preceq \nu$.
Clearly, this defines an equivalence relation on the set of nondecreasing functions
from $(0,\infty)$ to $[0,\infty)$.
It is easily checked that $\nu_{U,G} \sim \nu_{V,G}$
for any two relatively compact, open neighborhoods of unity contained in connected Lie groups.
Since we are only interested in equivalence classes,
we will therefore write $\nu_G = \nu_{U,G}$,  for a suitably chosen neighborhood $U$.

A group is called \textit{of polynomial growth} if $\nu_G(r) \preceq r^n$ for a suitable exponent $n \in [0,\infty)$.
We call $G$ \textit{of linear growth} if $\nu_G(r) \simeq r$ holds.
The relevance of growth for our arguments comes from the fact that it is a quasi-isometric invariant: 

\begin{lemma} \label{prop:growth_invariant}
    Let $G, H$ denote two connected Lie groups that are quasi-isometric.
    Then $\nu_G \simeq \nu_H$.
\end{lemma}

\begin{proof}
    Fix a symmetric, relatively compact neighborhood $U \subseteq G$ of unity,
    and denote by $d_U$ the associated word metric.
    By \cite[Propositions 3.D.23 and 3.D.29]{cornulier2016metric},
    there exists a growth function $\beta_G$ (cf. \cite[Definition 3.D.2]{cornulier2016metric}) that is (up to equivalence)
    defined purely in terms of $d_G$ and satisfies $\beta_G \sim \nu_G$.
    The same reasoning applies to $H$.
    Hence, since $G$ and $H$ are quasi-isometric, an application of \cite[Proposition 3.D.23]{cornulier2016metric} yields
    that $\beta_G \simeq \beta_H$, which then implies $\nu_G \simeq \nu_H$.
\end{proof}

We first give a characterization of nilpotent Lie groups with linear growth. 

\begin{lemma} \label{prop:nilp_lg}
    Let $G$ be a connected, simply connected nilpotent Lie group with linear growth.
    Then $G \cong \mathbb{R}$.
\end{lemma}

\begin{proof}
Since $G$ is simply connected and nilpotent, it follows from \cite[Corollary 2.9]{MR3267520} that it has strict polynomial growth,
in the sense that $\nu_G(r) \simeq r^{d(G)}$ for some $d(G) \ge \dim(G)$.
On the other hand, the assumption of linear growth implies $d(G) = 1$.
Since $G$ is nontrivial, we get $\dim(G)=1$, and the desired conclusion follows. 
\end{proof}

The following result is the main result of this section.

\begin{proposition} \label{prop:char_lin_grwth}
    Let $G$ denote a connected Lie group.
    Then the following assertions are equivalent:
    \begin{enumerate}
    \item[(i)] $G$ is quasi-isometric to $\mathbb{R}$.
    \item[(ii)] $G$ has linear growth.
    \item[(iii)] $G$ contains a  closed cocompact, noncompact one-parameter subgroup. 
    \end{enumerate}
\end{proposition}
\begin{proof}
    The implication $(i) \Rightarrow (ii)$ is due to Proposition \ref{prop:growth_invariant}. 

    The proof of $(ii) \Rightarrow (iii)$ relies on results from \cite{MR1825980,MR4079363}.
    Since $G$ is of polynomial growth, an application of \cite[Proposition 1]{MR1825980}
    yields the existence of a maximal compact normal subgroup $K_1$ of $G$,
    and shows that $G/K_1$ is a Lie group.
    By \cite[Proposition 4.C.12]{cornulier2016metric}, the quotient map $G \to G/K_1$ must be a quasi-isometry
    (if both $G$ and $G / K_1$ are equipped with the word metric associated to a compact generating set).
    Therefore, by \Cref{prop:growth_invariant}, $H := G/K_1$ is a connected Lie group
    of linear (in particular polynomial) growth without nontrivial compact normal subgroups. 

    By \cite[Theorem 2]{MR4079363}, there exists a topological embedding $\varphi: H \to N \rtimes K_2$,
    where $N$ is a simply connected, connected nilpotent Lie group and $K_2$ is a compact group,
    and such that $\varphi(H) \subseteq N \rtimes K_2$ is  closed and cocompact.
    By definition of a (topological) semidirect product, the map
    \[
      \psi :  N \rtimes K_2 \to K_2, \quad g = n k \mapsto k
      \qquad \text{where } n \in N, k \in K_2
    \]
    is a continuous homomorphism.
    Given that $H$ is connected, the image $\psi(\varphi(H)) \subseteq K_2$
    is contained in the connected component $K_2^\ast$ of the unit element of $K_2$.
    Thus, by replacing $K_2$ with $K_2^\ast$, we can assume that $K_2$ is connected.
    Note that for this, we use that $(N \rtimes K_2^\ast) / \varphi(H)$
    is compact by Lemma \Cref{lem:CompactQuotient}, because $\varphi(H) \subseteq N \rtimes K_2^\ast \subseteq N \rtimes K_2$
    are closed subgroups in $N \rtimes K_2$ and $(N \rtimes K_2) / \varphi(H)$
    is compact.
    
    Since $(N \rtimes K_2)/ \varphi(H)$ is compact, the inclusion map $\varphi(H) \hookrightarrow N \rtimes K_2$ is a quasi-isometry
    by \cite[Proposition 4.C.11]{cornulier2016metric} (again, with respect to suitable word metrics).
    The same reasoning applies to the inclusion map $N \hookrightarrow N \rtimes K_2$.
    Combining these observations yields that $G$ and $N$ are quasi-isometric.
    As such, it follows therefore from  \Cref{prop:growth_invariant} that $N$ has linear growth,
    and thus $N \cong \R$ by \Cref{prop:nilp_lg}.
    Hence, the automorphism group of $N$ is isomorphic to $\mathbb{R}^*$,
    which contains no nontrivial connected compact subgroup.
    Since the compact group $K_2$ acts continuously on $N$, this action must then be trivial.
    Hence, $N \rtimes K_2 = N \times K_2$.

   For constructing the one-parameter subgroup, let $\pi_N: N \times K_2 \to N$ denote the canonical projection.
    Since $\varphi(H)$ is cocompact in $N \times K_2$ and $N \cong \R$ is not compact,
    $\varphi(H)$ is not contained in $\{ e_N \} \times K_2$. 
     Since the image of the exponential map $\exp_H$ generates $H$,
    this entails the existence of an element $X$ of the Lie algebra of $H$ such that
    $(\pi_N \circ \varphi)(\exp_H(\mathbb{R}X))$ is nontrivial and connected.
    Since $N \cong \R$ does not have any nontrivial connected subgroups,
    it follows that $\pi_N(\varphi(\exp_H (\mathbb{R} X))) = N$.
    This entails that  $\varphi(\exp_H(\mathbb{R} X)) \subseteq \varphi(H) \subseteq N \times K_2$
    is cocompact in $N \times K_2$. 
    Next, note that $\varphi(\exp_H(\R X)) \subseteq N \times K_2$ is closed by \Cref{lem:OneParameterGroupClosedness},
    since its closure cannot be compact, because $\varphi(\exp_H(\R X))$ is cocompact in the
    noncompact group $N \times K_2$. By \Cref{lem:CompactQuotient}, it follows therefore that $\varphi(\exp_H(\mathbb{R} X))$ is also cocompact in $\varphi(H)$ 
    Since $\varphi$ is a topological embedding,
    it follows that $\exp_H(\mathbb{R} X)$ must be cocompact and closed in $H$. 
    
    It remains therefore to lift the one-parameter subgroup to $G$.
    To do this, let $q: G \to H$ denote the quotient map,
    and $dq: \mathfrak{g} \to \mathfrak{h}$ the differential map between the respective Lie algebras.
     Since $q$ is an open Lie group homomorphism,
    $dq$ is a surjective Lie algebra homomorphism, see, e.g., \cite[Proposition 9.2.13]{MR3025417}).
    In particular, there exists $Y \in \mathfrak{g}$ with $dq(Y) = X$,
    and $q(\exp_G( t Y)) = \exp_H(tX)$ holds for all $t \in \mathbb{R}$,  see, e.g., \cite[Proposition 9.2.10]{MR3025417}.
    In particular, since $\exp_H(\mathbb{R} X)$ is noncompact and closed, $\exp_G(\mathbb{R} Y)$ does not have compact closure.
    But then it is closed (by \Cref{lem:OneParameterGroupClosedness}) and noncompact, as well. For showing that $\exp_G(\mathbb{R}Y) \subseteq G$ is cocompact, note that since $\exp_H(\mathbb{R}X) \subseteq H$ is cocompact,
    there exists a compact set $\Theta \subseteq H$ with $H = \Theta \exp_H (\R X)$, cf. \cite[Lemma 2.46]{FollandAHA}.
    Similarly, there exists a compact set $\Omega \subseteq G$ with $\Theta = q(\Omega)$.
    Given $g \in G$, we can write $q(g) = \theta \cdot \exp_H (t X)$ for certain $\theta \in \Theta$ and $t \in \R$.
    Writing $\theta = q(\omega)$ with $\omega \in \Omega$, we then see
    $q(g) = q(\omega \exp_G (t Y))$, and thus $g \in \omega \exp_G (t Y) K_1 = \omega K_1 \exp_G (t Y)$,
    which shows that $G = \Omega K_1 \exp_G(\R Y)$, so that $G / \exp_G(\R Y)$ is compact. Thus, $\exp_G(\mathbb{R}Y) \subseteq G$ is cocompact, which shows (iii).

    The implication $(iii) \Rightarrow (i)$ follows again by \cite[Proposition 4.C.11]{cornulier2016metric}.
\end{proof}

\section*{Acknowledgements}
For J.~v.~V., this research was funded in whole or in part by the Austrian
Science Fund (FWF): 10.55776/J4555 and 10.55776/PAT2545623. For open access purposes, the author has applied a CC BY public copyright license to any author-accepted manuscript version arising from this submission. In addition, 
J.~v.~V.\ is grateful for the hospitality and support
of the Katholische Universität Eichstätt-Ingolstadt during his visit. 
F.V.\ acknowledges support by the Hightech Agenda Bavaria.

\bibliographystyle{abbrv}
\bibliography{bib}

\end{document}